\documentclass[11pt]{amsart}
\usepackage[all,cmtip]{xy}
\usepackage{enumerate, comment}
\usepackage{mathrsfs} 
\usepackage{ amssymb, latexsym, amsmath}
\usepackage{fullpage}
\usepackage{url}
\usepackage{xcolor}
\usepackage{hyperref}
\usepackage{helvet}
\usepackage{amsfonts}
\usepackage{tikz}
\usepackage{amssymb}
\usepackage{amsthm}
\usepackage{amsmath}
\usepackage{amscd}
\usepackage{stmaryrd}
\usepackage[mathscr]{eucal}
\usepackage{indentfirst}
\usepackage{graphicx}
\usepackage{tikz-cd}
\usepackage{graphics}
\usepackage{pict2e}
\usepackage{epic}
\usepackage[normalem]{ulem}
\usepackage[OT2,T1]{fontenc}
\numberwithin{equation}{section}
\usepackage[margin=3.4cm]{geometry}
 
\usepackage{color}
\DeclareSymbolFont{cyrletters}{OT2}{wncyr}{m}{n}
\DeclareMathSymbol{\Sha}{\mathalpha}{cyrletters}{"58}

\newcommand{\Z}{\mathbb{Z}}
\newcommand{\g}{\operatorname{Ad}^0\bar{\rho}}
\newcommand{\Q}{\mathbb{Q}}
\newcommand{\F}{\mathbb{F}}
\newcommand{\On}{\mathcal{O}_{\slash n}}
\newcommand{\op}[1]{\operatorname{#1}}
\newcommand\mtx[4] { \left( {\begin{array}{cc}
   #1 & #2 \\
   #3 & #4 \\
  \end{array} } \right)}

\theoremstyle{plain}
 \theoremstyle{definition}
\newtheorem{Th}{Theorem}[section]
\newtheorem{Lemma}[Th]{Lemma}

\newtheorem{Corollary}[Th]{Corollary}
\newtheorem{Example}[Th]{Example}
\newtheorem{Hypothesis}[Th]{Hypothesis}
\newtheorem{Proposition}[Th]{Proposition}
\newtheorem{Remark}[Th]{Remark}
 \theoremstyle{definition}
\newtheorem{Definition}[Th]{Definition}

\newtheorem{Fact}[Th]{Fact}
\begin{document}

\title{On the {\Large$\mu$}-invariants of residually reducible Galois representations}
\author{Anwesh Ray}
\address{Chennai Mathematical Institute, H1, SIPCOT IT Park, Kelambakkam, Siruseri, Tamil Nadu 603103, India}
  \email{anwesh@cmi.ac.in}
\author{R. Sujatha}
\address{Department of Mathematics \\ University of British Columbia \\
  Vancouver BC, V6T 1Z2, Canada.} 
  \email{sujatha@math.ubc.ca} 
  
\begin{abstract}
The Iwasawa $\mu$-invariant of the Selmer group of a residually reducible Galois representation arising from a Hecke eigencuspform is studied. Furthermore, certain Iwasawa-invariants refining the $\mu$-invariant are defined and analyzed. As an application, we show that given a reducible mod-$p$ Galois representation $\bar{\rho}$ and any choice of integer $N\geq 1$, there is a modular Galois representation lifting $\bar{\rho}$ whose associated Selmer group has $\mu$-invariant $\geq N$. This is a refinement of Serre's conjecture in the residually reducible case. \end{abstract}

\maketitle
\section{Introduction}
Iwasawa theory of Galois representations associated to elliptic curves and modular forms encodes deep arithmetic information. Greenberg analyzed the algebraic structure of the Selmer groups associated to ordinary Galois representations over the cyclotomic $\Z_p$-extension of a number field (see \cite{greenbergIT}). {An important invariant associated with the Selmer group is its $\mu$-invariant. It is} conjectured that when the residual representation is irreducible, then the Iwasawa $\mu$-invariant of $p$-primary Selmer group is zero, see \cite[Conjecture 1.11]{greenbergIWEC}.  Let $E$ be an elliptic curve over a number field $F$ and $\op{Sel}(E/F^{\op{cyc}})$ be the $p$-primary Selmer group over the cyclotomic $\Z_p$-extension of $F$. Mazur gave examples of elliptic curves $E$ for which the residual representation $E[p]$ is reducible as a Galois module and the $\mu$-invariant of $\op{Sel}(E/F^{\op{cyc}})$ is positive, see \cite[section 10]{mazur1}. Drinen showed that to every finite Galois submodule $\alpha$ of $E[p^{\infty}]$, there is an integer $\delta(\alpha)\geq 0$ which contributes to the $\mu$-invariant of $\op{Sel}(E/F^{\op{cyc}})$, see \cite[Theorem 2.1]{drinen}.
   \par Denote by $\op{G}_{\Q}$ the absolute Galois group $\op{Gal}(\bar{\Q}/\Q)$ and fix an embedding $\iota:\bar{\Q}\hookrightarrow \bar{\Q}_p$. Let  $f=\sum_{n\geq 1} a_n e^{2\pi i \tau}$ be a normalized Hecke eigencuspform. Associated to $f$ is the Galois representation $\rho_f:\op{G}_{\Q}\rightarrow \op{GL}_2(K)$, {where} $K$ is the finite extension of $\Q_p$ generated by $\iota(a_n)$ for positive integers $n$. Let $\mathcal{O}$ be the valuation ring of $K$ and let $T$ be a $\op{G}_{\Q}$-stable $\mathcal{O}$-lattice and set $\rho_T:\op{G}_{\Q}\rightarrow \op{GL}_2(\mathcal{O})$ to be the Galois representation on $T$. As is well known, any elliptic curve $E$ over $\Q$ coincides with an eigencuspform $f$ of weight $2$ with rational Fourier coefficients. The Galois representation on the $p$-adic Tate module of $E$ then coincides with the $p$-adic Galois representation associated to $f$. We prove our results for Selmer groups associated to $p$-ordinary Hecke eigencuspforms of weight $k\geq 2$. The Selmer groups considered here are defined over admissible, pro-$p$, $p$-adic Lie extensions $\Q_{\infty}$, with Galois group $\op{G}:=\op{Gal}(\Q_{\infty}/\Q)$  (see \cite[section 2]{raysujatha}). In this case, the Iwasawa algebra $\Lambda(\op{G}):=\mathcal{O}[[\op{G}]]$ is a left and right noetherian local domain. Further, it is known that $\Lambda(\op{G})$ is an Auslander regular ring \cite[Theorem 3.26]{venjakob}. There is a dimension theory associated to finitely generated modules over such rings.
   \par In this article, we will study the $\mu$-invariant of the $p$-primary Selmer group over $\Q_{\infty}$. Furthermore, certain refinements of the $\mu$-invariant are defined and studied (see Definition $\ref{refinedmudef}$). Throughout, it is assumed that the residual representation $\bar{\rho}_T:\op{G}_{\Q}\rightarrow \op{GL}_2(\mathcal{O}/\varpi)$ is reducible. Note that the $\mu$-invariant depends on the choice of the $\op{G}_{\Q}$-stable lattice $T$. When the Galois representation arises from an elliptic curve, this lattice is taken to be the $p$-adic Tate module. Our results show that the $\mu$-invariant depends crucially on the residual representation. We classify the residual representation $\bar{\rho}_T$ into two types: \textit{aligned} or \textit{skew} (see Definition $\ref{aligneddef}$). We prove results on the $\mu$-invariants in the two mutually disjoint cases. The first example of an elliptic curve with positive $\mu$-invariant was provided by Mazur \cite[Example 2]{mazur1}. This example is residually aligned. It is shown that if $\rho$ is residually reducible and aligned, then  $\mu$ is positive (cf. Theorem $\ref{theorem1}$), furthermore, it is shown that the $\mu$-multiplicity is $1$ (cf. Definition $\ref{refinedmudef}$ and Theorem $\ref{theorem2}$).
   \par Here is a brief outline of how our results fit into a broader framework of known results and conjectures on the $\mu$-invariants of Selmer groups of elliptic curves. Greenberg conjectured that any elliptic curve over $\Q$ is isogenous to one for which the $\mu$-invariant of the Selmer group is zero. It is shown that an elliptic curve whose Galois representation is residually reducible and aligned is isogenous to one whose Galois representation is residually reducible and skew, see Proposition $\ref{alignedskew}$. Indeed, the $\mu$-invariant is positive in the residually aligned case, and in the residually skew case, examples indicate that the $\mu$-invariant vanishes. An explicit example is given in section $\ref{section 7}$ which illustrates this. This can be viewed as a justification for Greenberg's conjecture. More generally, we study the case when $\rho_{T/n}:=\rho_T\mod{\varpi^n}$ is \textit{reducible} for an integer $n>0$. This means that there are characters $\varphi_{i,n}:\op{G}_{\Q}\rightarrow \op{GL}_1(\mathcal{O}/\varpi^n)$ lifting $\varphi_i$ for $i=1,2$, such that $\rho_{T/n}\simeq \mtx{\varphi_{1,n}}{\ast}{0}{\varphi_{2,n}}$. It is shown that if $\bar{\rho}_T$ is aligned and $n>0$ is such that $\rho_{T/n}$ is reducible, then the $\mu$-invariant of the $p$-primary Selmer group of $\rho_T$ is at least $n$. In practice, it is hard to find examples of Galois representations arising from elliptic curves which are reducible modulo $p^2$. However, there are modular forms for which the associated Galois representation is reducible modulo a high power of the uniformizer, for instance, the proof of Theorem $\ref{mainlifting}$ gives an explicit construction of modular Galois representations that are reducible modulo a high power of the uniformizer. We also prove more precise results for the refined $\mu$-invariants (see Theorem $\ref{theorem1}$).
\par Here is an interesting application of our analysis of $\mu$-invariants to the deformation theory of Galois representations. Applying the results of Skinner and Wiles \cite{skinnerwiles}, Hamblen and Ramakrishna in \cite{hamblenramakrishna} proved Serre's conjecture for residually reducible Galois representations which satisfy some additional conditions. Let $\F$ be a finite field of characteristic $p$. The ring of Witt vectors $\op{W}(\F)$ is the valuation ring of the unique unramified extension of $\Q_p$ with residue field $\F$. Denote by $\bar{\chi}$ the mod-$p$ cyclotomic character. A character $\kappa:\op{G}_{\Q}\rightarrow \op{GL}_1(\F)$ is \textit{odd} if $\kappa(c)=-1$, where $c\in \op{G}_{\Q}$ denotes complex conjugation. The character $\kappa$ is \textit{even} otherwise.
\begin{Th}\label{mainlifting}
Let $N$ be a non-negative integer. Let $\bar{\rho}:\op{G}_{\Q}\rightarrow \op{GL}_2(\F)$ be a Galois representation such that $\bar{\rho}\simeq \mtx{\varphi}{\ast}{0}{1}$ for a character $\varphi:\op{G}_{\Q}\rightarrow \op{GL}_1(\F)$. Assume that the following conditions are satisfied:
\begin{enumerate}
\item $p\geq 5$, $\varphi$ is odd,
    \item $\bar{\rho}$ is indecomposable,
    \item $\varphi^2\neq 1,\varphi\neq \bar{\chi}^{-1},\varphi_{\restriction \op{G}_{\Q_p}}\neq 1, \varphi_{\restriction \op{G}_{\Q_p}}\neq\bar{\chi}_{\restriction \op{G}_{\Q_p}}$ and the image of $\varphi$ spans $\F$ over $\F_p$.
\end{enumerate}
Let $k$ be a positive integer such that $k\equiv 2\mod{p^N(p-1)}$. Then there is a Galois representation\[\rho_T:\op{G}_{\Q}\rightarrow \op{GL}_2(\op{W}(\F))\] which lifts $\bar{\rho}$ such that:
\begin{enumerate}[(i)]
\item $\rho_T$ lifts $\bar{\rho}$, i.e., $\bar{\rho}_T\simeq \bar{\rho}$,
    \item $\rho_T$ arises from a $p$-ordinary Hecke eigencuspform $f$ of weight $k$ and a choice of $\op{G}_{\Q}$-stable lattice $T$ which is residually aligned,
    \item the $p$-primary Selmer group associated to $\rho_T$ has $\mu\geq N$.
\end{enumerate}
\end{Th}
Thus one may not only lift a Galois representation $\bar{\rho}$ to one arising from a modular form, but one can also arrange the $\mu$-invariant to be as large as possible. The above result is a refinement of Serre's conjecture in the residually reducible case and significantly extends the main result of Hamblen-Ramakrishna \cite{hamblenramakrishna}. It crucially relies on our results on $\mu$-invariants proved in sections $\ref{section 3}$ and $\ref{section 4}$ of this paper. We emphasize that our methods provide an explicit lift of the actual representation $\bar{\rho}$, rather than that of the semisimplification of $\bar{\rho}$. This is a stronger statement altogether and is indeed crucial when it comes to showing that the lift has $\mu$-invariant $\geq N$, since the nature of the $\mu$-invariant depends on the actual residual representation $\bar{\rho}$. As a result, we do not work with pseudo-representations, and instead resort to a purely Galois theoretic construction.
\par The lifting theorem above may be contrasted with the result of Bella\"iche and Pollack \cite{bellaichepollack}, who study the variation of $\mu$-invariants in a Hida family of tame level $N=1$, hence, for Galois representations unramified away from $\{p\}$. Our approach relies on a purely Galois theoretic construction and we make no assumption on the set of primes at which $\bar{\rho}$ is ramified.
\par The paper is organized into 7 sections. Certain preliminary notions are discussed in section $\ref{section 2}$. It is in this section that the refined $\mu$-invariants are introduced. In sections $\ref{section 3}$ and $\ref{section 4}$, results are proved that characterize the $\mu$-invariants of Selmer groups associated to residually reducible Galois representations arising from Hecke eigencuspforms. In section $\ref{section 5}$, necessary preparations are made for the proof of Theorem $\ref{mainlifting}$. In section $\ref{section 6}$, Theorem $\ref{mainlifting}$ is proved. Examples are discussed in section $\ref{section 7}$.
\newline\textit{Acknowledgements:} The authors would like to thank Ravi Ramakrishna for helpful comments. The second named author gratefully acknowledges support from NSERC Discovery grant 2019-03987.
\section{Preliminaries}\label{section 2}
\par Throughout, $p$ is a fixed odd prime number and $\iota:\bar{\Q}\hookrightarrow \bar{\Q}_p$ is a fixed embedding. Let $f$ be a $p$-ordinary Hecke eigencuspform of weight $k\geq 2$ on $\Gamma_1(N)$. Associated to $f$ is the Galois representation $\rho_f:\op{G}_{\Q}\rightarrow \op{GL}_2(K)$, where $K$ is a finite extension of $\Q_p$. The representation $\rho_f$ is continuous, irreducible and unramified at primes $l\nmid Np$. Let $V$ be the underlying $K$ vector space on which $\rho_f$ acts and $T\subset V$ be a choice of a $\op{G}_{\Q}$-stable $\mathcal{O}$-lattice. The Galois action on $T$ is encoded by the integral representation: \[\rho_T: \op{G}_{\Q}\rightarrow \op{Aut}(T)\xrightarrow{\sim} \op{GL}_2(\mathcal{O}).\] Set $\bar{\rho}_T$ to denote the mod-$\varpi$ reduction of $\rho_T$. The Brauer-Nesbitt theorem implies that the semisimplification of $\bar{\rho}_T$ is independent of the lattice $T$, i.e., if $T$ and $T'$ are Galois-stable lattices in $V$, then the semisimplifications $\bar{\rho}_T^{\op{ss}}$ and $\bar{\rho}_{T'}^{\op{ss}}$ are isomorphic. In particular, if $\bar{\rho}_T$ is reducible, then so is $\bar{\rho}_{T'}$. However, it is not the case that $\bar{\rho}_T$ is isomorphic to $\bar{\rho}_{T'}$. Since the Selmer group is associated to the integral representation $\rho_T$, the choice of the lattice $T$ plays a role in the study of Iwasawa invariants. Suppose that $E$ and $E'$ are elliptic curves with Tate modules $T:=T_p(E)$ and $T':=T_p(E')$ and set $V:=T\otimes \Q_p$ (resp. $V':=T'\otimes \Q_p$). Then $E$ and $E'$ are isogenous if and only if there is an isomorphism of Galois representations $\alpha:V'\xrightarrow{\sim} V$. In this case, we may identify $V'$ with $V$ and thus, $T$ and $T'$ are both $\op{G}_{\Q}$-stable lattices in $V$. This motivates the following definition.
\begin{Definition}
Let $T$ and $T'$ be $\mathcal{O}$-lattices on which $\op{G}_{\Q}$ acts by continuous $\mathcal{O}$-linear automorphisms. Then, $T$ and $T'$ are said to be isogenous if $V:=T\otimes_{\mathcal{O}}K $ and $V':=T'\otimes_{\mathcal{O}} K$ are isomorphic as Galois representations. In particular, if $V$ if the underlying vector space of $\rho_f$, then any two $\op{G}_{\Q}$-stable $\mathcal{O}$-lattices $T$ and $T'$ in $V$ are isogenous.
\end{Definition}Let $A$ denote the $p$-divisible group $V/T$. Given an abelian group $M$, the $p^r$-torsion subgroup of $M$ is denoted $M[p^r]$, and set $M[p^{\infty}]:=\bigcup_{r\geq 1} M[p^r]$. Note that $A[p]$ is isomorphic to the residual representation $\bar{\rho}_T$. Note that when the Galois representation arises from an elliptic curve $E$ over $\Q$ and $T$ is the Tate module of $E$, then $A$ is identified with $E[p^{\infty}]$.
\par Fix a pro-$p$, admissible, $p$-adic Lie extension $\Q_{\infty}$ of $\Q$ and set $\op{G}:=\op{Gal}(\Q_{\infty}/\Q)$. Choose a finite set $S$ of prime numbers containing the primes that ramify in $\Q_{\infty}$ and the primes at which $\rho_T$ is ramified. Let $\chi:\op{G}_{\Q}\rightarrow \op{GL}_1(\Z_p)$ denote the $p$-adic cyclotomic character. For ease of notation, we shall simply denote the restriction of $\chi$ to $\op{G}_{\Q_p}$ by $\chi$. Since $f$ is assumed to be ordinary at $p$, the local Galois representation $\rho_{T\restriction \op{G}_{\Q_p}}$ is ordinary. Hence $T$ fits into a short exact sequence of $\op{G}_{\Q_p}$-modules
\begin{equation}\label{Tpm}0\rightarrow T^+\rightarrow T \rightarrow T^-\rightarrow 0,\end{equation} where $T^+\simeq \mathcal{O}(\chi^{k-1} \gamma_1)$ and $T^-\simeq \mathcal{O}( \gamma_2)$ for some finite order unramified characters $\gamma_1,\gamma_2: \op{G}_{\Q_p}\rightarrow \op{GL}_1(\mathcal{O})$ (see \cite[p. 131]{greenbergIT}). Set $A^{\pm}$ to denote the $p$-divisible group given by $T^{\pm}\otimes_{\Z_p}\Q_p/\Z_p$.
\begin{Lemma}\label{unique}
With notation as above, the choice of $T^+$ and $T^-$ is unique.
\end{Lemma}
\begin{proof}
Suppose that $\{e_1, e_2\}$ is an $\mathcal{O}$-module basis for $T$ for which $T^+=\mathcal{O} e_1$. With respect to this basis, the local Galois representation is given by 
\[\rho_{T}|_{\op{G}_{\Q_p}}=\mtx{\chi^{k-1}\gamma_1 }{\alpha}{0}{\gamma_2}.\]Therefore, for $g\in \op{G}_{\Q_p}$, we have that \[g e_1=\chi^{k-1}(g)\gamma_1(g) e_1\text{ and }g e_2=\alpha(g) e_1+\gamma_2(g) e_2.\] Consider another choice of $\hat{T}^+$ and $\hat{T}^-$ and let $\{\hat{e}_1, \hat{e}_2\}$ be an $\mathcal{O}$-lattice basis of $T$ such that $\hat{T}^+=\mathcal{O}\hat{e}_1$. With respect to the basis $\{\hat{e}_1, \hat{e}_2\}$ we have
\[\rho_T|_{\op{G}_{\Q_p}}=\mtx{\chi^{k-1}\hat{\gamma}_1}{\hat{\alpha}}{0}{\hat{\gamma}_2}.\]Writing $\hat{e}_1=a\hat{e}_1+b \hat{e}_2$, one obtains
\[\begin{split}& g \hat{e}_1=\chi^{k-1}(g) \hat{\gamma}_1(g) \hat{e}_1=a \chi^{k-1}(g) \hat{\gamma}_1(g) e_1+b \chi^{k-1}(g) \hat{\gamma}_1(g)e_2,\\
& g \hat{e}_1=a g e_1+b g e_2=(a \chi^{k-1}(g)\gamma_1(g)+b\alpha(g)) e_1+b \gamma_2(g) e_2.
\end{split}\]
Comparing the coefficients of $e_2$, we have that 
\begin{equation}\label{eq1}b \chi^{k-1}(g) \hat{\gamma}_1(g)=b \gamma_2(g).\end{equation} Let $g$ be in the inertia group of $\op{G}_{\Q_p}$ such that $\chi^{k-1}(g)\neq 1$. Since $\gamma_i$ and $\hat{\gamma}_i$ are unramified, it follows that $\gamma_i(g)=\hat{\gamma}_i(g)=1$ for $i=1,2$. By $\eqref{eq1}$, one obtains the relation $b(\chi^{k-1}(g)-1)=0$. From this, it follows that $b=0$ and $e_1=a\hat{e}_1$. Therefore $\hat{T}^+$ is contained in $T^+$. Note that $\hat{T}^-=T/\hat{T}^+$ contains the finite $\mathcal{O}$-submodule $T^+/\hat{T}^+$. Since $\hat{T}^+\simeq \mathcal{O}$ as an $\mathcal{O}$-module, the only finite $\mathcal{O}$-submodules are trivial. Therefore, we have shown that $T^+=\hat{T}^+$ and consequently $T^-=\hat{T}^-$ as well. Therefore, the pair $(T^+,T^-)$ is unique.
\end{proof}

\par We next recall the definition of the $p$-primary Selmer group $\op{Sel}(A/\Q_{\infty})$. The reader may refer to the discussion on \cite[p.98]{greenbergIT} for further details. Let $L$ be a number field contained in $\Q_{\infty}$. Set $L^{\op{cyc}}$ to denote the cyclotomic $\Z_p$-extension of $L$. For a prime $\eta$ of $L^{\op{cyc}}$, the completion $L^{\op{cyc}}_{\eta}$ is the union of completions at $\eta$ of all number fields contained in $L^{\op{cyc}}$. Let $\op{I}_{\eta}$ denote the inertia group of $\op{Gal}(\bar{L}_{\eta}^{\op{cyc}}/L_{\eta}^{\op{cyc}})$. At each prime $v\in S\backslash \{p\}$, set
\[\mathcal{H}_v(A/L^{\op{cyc}}):=\bigoplus_{\eta|v} \op{im}\left\{H^1(L^{\op{cyc}}_{\eta}, A)\longrightarrow H^1(\op{I}_{\eta}, A) \right\},\] where $\eta$ runs through the finite set of primes of $L^{\op{cyc}}$ that lie above $v$. Set
\[\mathcal{H}_p(A/L^{\op{cyc}}):=\bigoplus_{\eta|p} \op{im}\left\{H^1(L^{\op{cyc}}_{\eta}, A)\longrightarrow H^1(\op{I}_{\eta}, A^-) \right\},\] where $\eta$ runs through the primes of $L^{\op{cyc}}$ above $p$. For $v\in S$, denote by $\mathcal{H}_v(A/\Q_{\infty})$ the direct limit with respect to restriction maps
\[\mathcal{H}_v(A/\Q_{\infty}):=\varinjlim_L \mathcal{H}_v(A/L^{\op{cyc}}),\]where $L$ runs through all number fields contained in $\Q_{\infty}$. The Selmer group $\op{Sel}(A/\Q_{\infty})$ is the kernel of the localization map
\begin{equation}\label{lambdaS}\lambda_{S}:H^1\left(\Q_S/\Q_{\infty},A\right)\rightarrow \bigoplus_{v\in S} \mathcal{H}_v(A/\Q_{\infty})\end{equation}and is independent of the choice of $S$.
 Put $\Gamma:=\op{Gal}(\Q^{\op{cyc}}/\Q)$ and identify the Iwasawa algebra $\mathcal{O}[[\Gamma]]$ with the formal power series ring $\mathcal{O}[[T]]$ after making a choice of a topological generator $\gamma$ of $\Gamma$ and letting $T$ be equal to $\gamma-1$. Given a discrete, $p$-primary module $\rm{M}$ over $\Lambda$, set $\rm{M}^{\vee}:=\op{Hom}_{\op{cnts}}(\rm{M}, \Q_p/\Z_p)$ to denote the Pontryagin dual of $\rm{M}$. The dual Selmer group $\op{Sel}(E/\Q_{\infty})^{\vee}$ is a cofinitely generated $\Lambda(\op{G})$-module. Let $Q(\op{G})$ be the skew field of fractions of $\Lambda(\op{G})$. A module $\rm{M}$ over $\Lambda(\op{G})$ is torsion if the dimension of $\rm{M}\otimes_{\Lambda(\op{G})}Q(\op{G})$ over the skew field $Q(\op{G})$ is zero. It is necessary to make the following hypothesis on the Selmer group. 
 \begin{Hypothesis}
 Assume that $\op{Sel}(A/\Q_{\infty})^{\vee}$ is a torsion $\Lambda(\op{G})$-module.
 \end{Hypothesis} Over the cyclotomic $\Z_p$-extension, this is a well known conjecture of Mazur, and conjectured to be true more generally over an admissible $p$-adic Lie extension (see \cite[Conjecture 1.3]{coateshowson}). Let $\mathcal{PN}$ be the category of finitely generated pseudonull submodules over $\Lambda(\op{G})$.
 The structure theory of modules over $\Lambda(\op{G})$ allows us to define refined $\mu$-invariants.
 \begin{Definition}\label{refinedmudef}  Let $\rm{M}$ be a finitely generated torsion $\Lambda(\op{G})$-module, there is a decomposition
\[\rm{M}[\varpi^{\infty}]\simeq \bigoplus_{i=1}^t \left(\Lambda(\op{G})/\varpi^i\right)^{\mu_i}\mod{\mathcal{PN}},\]with $\mu_t>0$ (cf. \cite[Proposition 3.11]{Howsoncentraltorsion} and \cite[Theorem 3.40]{venjakob}).
\begin{itemize}
    \item The \textit{$\mu$-invariant} is defined by \[\mu(\rm{M}):=\begin{cases}\sum_i i \mu_i &\text{ if }t>0,\\
    0&\text{ if }t=0.\end{cases}\]
\item The vector $\vec{\mu}(\rm{M})=(\mu_1,\dots, \mu_t)$ is uniquely determined and is called the \textit{$\mu$-vector} of $\rm{M}$. If $\mu(\rm{M})=0$ then we set $\vec{\mu}(\rm{M}):=(0)$, i.e., the vector with one summand whose entry is $0$.
\item The number $t=t(\rm{M})$ is the greatest power of $\varpi$ in the decomposition and is referred to as the \textit{$\mu$-exponent}.
\item The number of cyclic summands is called the \textit{$\mu$-multiplicity}, given by $\op{r}(\rm{M}):=\mu_1+\dots +\mu_t$.
\end{itemize}
 \end{Definition}
 When $\rm{M}=\op{Sel}(A/\Q_{\infty})^{\vee}$, set $\tau(A/\Q_{\infty}):=\tau(\rm{M})$ for $\tau\in \{\mu, \vec{\mu}, t, \op{r}\}$. We point out that this notation for the $\mu$-multiplicity $\op{r}(M)$ is from \cite{perbet}, where it has been studied.
 \section{Results on positive {\large$\mu$}-invariants}\label{section 3}
 \par In this section, we prove some of the main results on $\mu$-invariants. Let $K$ be a finite extension of $\Q_p$ and $\mathcal{O}$ denote the valuation ring of $K$ with uniformizer $\varpi$. Let $T$ be an $\mathcal{O}$-lattice on which the Galois group $\op{G}_{\Q}$ acts via a modular Galois representation $\rho_T:\op{G}_{\Q}\rightarrow \op{GL}_2(\mathcal{O})$, and $T^{+}$ and $T^-$ are defined as in $\eqref{Tpm}$. Note that by Lemma $\ref{unique}$, the choice of $T^+$ and $T^-$ are unique. Choose an ordered $\mathcal{O}$-lattice basis $\mathcal{B}=\{e_1,e_2\}$ of $T$ so that $T^+=\mathcal{O}e_1$. Therefore, with respect to $\mathcal{B}$, the local representation is of the form: \[\rho_{T| \op{G}_{\Q_p}}=\mtx{\gamma_1 \chi^{k-1}}{\ast}{0}{\gamma_2},\] where $\gamma_1$ and $\gamma_2$ are finite order unramified characters. Let $\bar{e}_i$ denote the $\varpi$-reduction of $e_i$ for $i=1,2$ and set $\bar{\mathcal{B}}:=\{\bar{e}_1, \bar{e}_2\}$.
 \begin{Definition}\label{aligneddef}
  There are $2$ separate cases to consider:
  \begin{enumerate}
      \item $\rho_T$ is said to be \textit{residually aligned} if $T^+/\varpi$ is a $\op{G}_{\Q}$-stable submodule of $T/\varpi$ on which the Galois action is via an odd character $\varphi_1$. With respect to the mod-$\varpi$ basis $\bar{\mathcal{B}}$, we have that 
 \[\bar{\rho}_T=\mtx{\varphi_1}{\ast}{0}{\varphi_2}.\]
 \item $\rho_T$ is said be \textit{residually skew} if it is not residually aligned. This is the case if either:
 \begin{itemize}
     \item $T^+/\varpi$ is a $\op{G}_{\Q}$-stable submodule of $T/\varpi$ on which $\op{G}_{\Q}$ acts by an even character $\varphi_1$, and w.r.t the basis $\bar{\mathcal{B}}$,  
 \[\bar{\rho}_T=\mtx{\varphi_1}{\ast}{0}{\varphi_2}.\]
     \item $T^+/\varpi$ is not a $\op{G}_{\Q}$-stable submodule of $T/\varpi$. After replacing $e_2$ by $e_2+k e_1$ we may assume that $\F\bar{e}_2$ is a $\op{G}_{\Q}$-stable submodule, and we have that $\bar{\rho}_T$ is indecomposable of the form
 \[\bar{\rho}_T=\mtx{\varphi_1}{0}{\ast}{\varphi_2}.\]
 \end{itemize} 
  \end{enumerate}
 \end{Definition}
The following Proposition shows that for any residually aligned lattice $T$, there is an isogenous lattice $T'$ which is residually skew.
\begin{Proposition}\label{alignedskew}
Let $T$ be an $\mathcal{O}$-lattice equipped with an $\mathcal{O}$-linear $\op{G}_{\Q}$-action, and let $\rho_T:\op{G}_{\Q}\rightarrow \op{GL}_2(\mathcal{O})$ denote the integral representation on $T$. Assume that the the residual representation $\bar{\rho}_T:\op{G}_{\Q}\rightarrow \op{GL}_2(\F)$ is reducible and aligned. Then there is an isogenous $\mathcal{O}$-lattice $T'$ such that the Galois representation $\rho_{T'}$ is residually skew.
\end{Proposition}
\begin{proof}
Let $\mathcal{B}=\{e_1,e_2\}$ be a basis of $T$ such that $T^+=\mathcal{O}e_1$. With respect to the choice of basis, $\rho_T$ has an integral representation 
\[\rho_T=\mtx{a}{b}{c}{d},\]and $a,b,c,d:\op{G}_{\Q}\rightarrow \mathcal{O}$ are the matrix coefficients for $\rho_T$. Note that since the residual representation $\bar{\rho}_T$ is reducible and aligned, it follows that $\varpi$ divides $c$. Let $m_1\geq 1$ be such that $\varpi^{m_1}||c$ and set $D:=\mtx{\varpi}{0}{0}{1}$. Denote by $V$ the vector space $T\otimes_{\mathcal{O}} K$ and $T' \subset V$ be the lattice $D^{-m_1} T$ with $\mathcal{O}$-basis $\{\varpi^{-m_1} e_1, e_2\}$. Note that the Galois representation on $T'$ is given by \[\rho_{T'}=D^{m_1}\rho_T D^{-m_1}=\mtx{a'}{b'}{c'}{d'}\] whith matrix coefficients $a' =a, b' =\varpi^{m_1} b, c' =\varpi^{-m_1} c$ and $d'=d$. Since $\chi$ is ramified and $\gamma$ is an unramified character $\op{G}_{\Q_p}\rightarrow \op{GL}_1(\mathcal{O})$, it is easy to show that $(T')^+=\mathcal{O}\cdot (\varpi^{-m_1}e_1)$. Since $\varpi\nmid c'$, it follows that $(T')^+/\varpi$ is not a $\op{G}_{\Q}$-stable submodule of $T'/\varpi$ and therefore, $\rho_{T'}$ is residually skew.
\end{proof}
 
 \par Throughout this section, assume that $\rho_T$ is residually reducible and aligned as in Definition $\ref{aligneddef}$. Let $\mathcal{B}=\{e_1, e_2\}$ be the basis of $T$ w.r.t which $T^+=\mathcal{O}e_1$ and let $\bar{\mathcal{B}}=\{\bar{e}_1, \bar{e}_2\}$ its mod-$\varpi$ reduction. With respect to $\bar{\mathcal{B}}$, the residual representation is given by
 \[\bar{\rho}_T=\mtx{\varphi_1}{\ast}{0}{\varphi_2}\] for an odd character $\varphi_1$. For $n>0$, let $\mathcal{B}_{/n}$ denote the mod-$\varpi^n$ reduction of $\mathcal{B}$ and $\rho_{T/n}:=\rho_T\mod{\varpi^n}$.
 \begin{Definition}
 Assume that $\rho_T$ is residually aligned and that $n>0$. Then, $\rho_{T/n}$ is said to be aligned if $T^+/\varpi^n$ is $\op{G}_{\Q}$-stable. The \textit{degree of alignment} for $\rho_T$ is the largest value of $n>0$ such that $\rho_{T/n}$ is aligned.
 \end{Definition} 
 In order to simplify notation, set $\On:=\mathcal{O}/\varpi^n$ and $\Lambda_n(\op{G}):=\On[[\op{G}]]$. Let $\Omega(\op{G}):=\F[[\op{G}]]$, thus $\Lambda_1(\op{G})=\Omega(\op{G})$. Note that if $\rho_{T/n}$ is aligned, then there are characters $\varphi_{i,n}:\op{G}_{\Q}\rightarrow \op{GL}_1(\mathcal{O}_{/n})$ for $i=1,2$ lifting $\varphi_i$ such that with respect to the mod-$\varpi^n$ reduction of the basis $\mathcal{B}$, we have that
 \[\rho_{T/n}=\mtx{\varphi_{1,n}}{\ast}{0}{\varphi_{2,n}}.\]
 Note that since $\varphi_1$ is odd, so is $\varphi_{1,n}$.
 For any character $\delta:\op{G}_{\Q}\rightarrow \op{GL}_1(\mathcal{O})$, denote by $\delta_n$ the mod-$\varpi^n$ reduction of $\delta$. For any character $\alpha:\op{G}_{\Q}\rightarrow \op{GL}_1(\F)$, denote by \[\tilde{\alpha}:\op{G}_{\Q}\rightarrow \op{GL}_1(\mathcal{O})\] the Teichm\"uller lift of $\alpha$.
 \begin{Definition}\label{liftabledefchar}
 Let $n$ be a positive integer. A character $\kappa: \op{G}_{\Q}\rightarrow \op{GL}_1(\mathcal{O}_{/n})$ is said to be \textit{liftable} if may be expressed as a product $\kappa=\chi_n^i \tilde{\alpha}_n$, where $i$ is an integer and $\tilde{\alpha}_n$ is the mod-$\varpi^n$ reduction of the Teichm\"uller lift $\tilde{\alpha}$ of a character
\[\alpha:\op{G}_{\Q}\rightarrow \op{GL}_1(\F).\]
 \end{Definition}
 A liftable character $\kappa=\chi_n^i \tilde{\alpha}_n$ exhibits a lift to characteristic zero, namely, $\chi^i \tilde{\alpha}$. Note that when $n=1$, $\kappa$ is always liftable. For residually aligned representations, the following result shows that the larger the degree of alignment, the larger the $\mu$-invariant.

 \begin{Th}\label{theorem1}
Let $n>0$ be such that $\rho_{T/n}$ is aligned and $\varphi_{1,n}$ is liftable in the sense of Definition $\ref{liftabledefchar}$. Then the $\mu$-exponent $t(A/\Q_{\infty})$ is $\geq n$. In particular, this implies that the $\mu$-invariant $\mu(A/\Q_{\infty})$ is also $\geq n$.
 \end{Th}
 We also prove results on the $\mu$-multiplicity. Let $\mathcal{L}_{\infty}$ be the maximal unramified abelian pro-$p$ extension of $\Q_{\infty}(\varphi_2)$ and let $X_{\infty}:=\op{Gal}(\mathcal{L}_{\infty}/\Q_{\infty}(\varphi_2))$. Note that $X_{\infty}$ is a cofinitely generated $\Lambda(\op{G})$-module. When $\Q_{\infty}$ is the cyclotomic $\Z_p$-extension, Ferrero and Washington \cite{ferrerowashington} proved that it has $\mu$-invariant equal to zero.
 \begin{Th}\label{theorem2}
Assume that $\rho_T$ is a residually reducible and that $X_{\infty}$ is a cotorsion $\Lambda(\op{G})$-module with $\mu=0$. Then the $\mu$-multiplicity $\op{r}(A/\Q_{\infty})$ is $\leq 1$. Moreover, if $\rho_T$ is residually aligned, then $\op{r}(A/\Q_{\infty})=1$.
 \end{Th}
\par Theorem $\ref{theorem1}$ will be proved at the end of this section and Theorem $\ref{theorem2}$ in the next section. We define a residual Selmer group associated to $A[\varpi^n]$, by defining local conditions at each prime $v\in S$. Let $L$ be a number field contained in $\Q_{\infty}$. For $v\in S\backslash \{p\}$, define \[\mathcal{H}_v(A[\varpi^n]/L^{\op{cyc}}):=\bigoplus_{\eta|v} \op{im}\left\{H^1(L^{\op{cyc}}_{\eta}, A[\varpi^n])\longrightarrow H^1(\op{I}_{\eta}, A[\varpi^n]) \right\},\]where $\eta$ ranges over the primes of $L^{\op{cyc}}$ such that $\eta|v$. Suppose that $v=p$ and $\eta|p$ a prime of $L^{\op{cyc}}$. Recall that $\op{I}_{\eta}$ is the inertia group at $\eta$. The local condition at $p$ is defined as follows: \[\mathcal{H}_p(A[\varpi^n]/L^{\op{cyc}}):= \bigoplus_{\eta|p} \op{im}\left\{H^1(L^{\op{cyc}}_{\eta}, A[\varpi^n])\longrightarrow H^1(\op{I}_{\eta}, A[\varpi^n])\right\},\] where $\eta$ ranges over the primes of $L^{\op{cyc}}$ such that $\eta|p$.
For $v\in S$, denote by $\mathcal{H}_v(A[\varpi^n]/\Q_{\infty})$ the direct limit with respect to restriction maps
\[\mathcal{H}_v(A[\varpi^n]/\Q_{\infty}):=\varinjlim_L \mathcal{H}_v(A[\varpi^n]/L^{\op{cyc}}),\]where $L$ runs through all number fields contained in $\Q_{\infty}$. The mod-$\varpi^n$ residual Selmer group is the defined as follows:
\[\op{Sel}(A[\varpi^n]/\Q_{\infty}):=\op{ker}\left(H^1(\Q_S/\Q_{\infty}, A[\varpi^n])\rightarrow \bigoplus_{v\in S}\mathcal{H}_v(A[\varpi^n]/\Q_{\infty}) \right).\] Note that it is a module over $\Lambda_n(\op{G})$.

 \begin{Proposition}\label{prop32}
 Suppose that the mod-$\varpi^n$ Selmer group $\op{Sel}(A[\varpi^n]/\Q_{\infty})$ has positive corank as a $\Lambda_n(\op{G})$-module. Then $\mu(A/\Q_{\infty})\geq n$ and $t(A/\Q_{\infty})\geq n$. 
\end{Proposition}

\begin{proof}
Note that we have an isomorphism\[\op{Sel}(A/\Q_{\infty})[\varpi^n]^{\vee} \simeq \frac{\op{Sel}(A/\Q_{\infty})^{\vee}}{\varpi^n \op{Sel}(A/\Q_{\infty})^{\vee}}.\]Recall that $\mathcal{PN}$ is the category of pseudonull $\Lambda(\op{G})$-modules. Setting $\rm{M}:=\op{Sel}(A/\Q_{\infty})^{\vee}$, we have that \begin{equation}\label{equation1}\rm{M}[\varpi^{\infty}]\simeq \bigoplus_i \left(\Lambda(\op{G})/\varpi^i\right)^{\mu_i}\mod{\mathcal{PN}}.\end{equation}
As $\rm{M}$ is assumed to be a torsion $\Lambda(\op{G})$-module, it follows from the proof of \cite[Lemma 2.4.1]{mflim} that 
\begin{equation}\label{strtheory}\rm{M}/\varpi^n \simeq \left(\Lambda(\op{G})/\varpi^{\op{min}\{i,n\}}\right)^{\mu_i}\mod{\mathcal{PN}}.\end{equation} Therefore, $\rm{M}/\varpi^n$ has positive
$\Lambda_n(\op{G})$-rank if and only if $t(A/\Q_{\infty})\geq n$. Moreover, if the invariant $t(A/\Q_{\infty})\geq n$, then it is clear that $\mu(A/\Q_{\infty})\geq n$. The Kummer sequence
\[0\rightarrow A[\varpi^n]\rightarrow A\xrightarrow{\varpi^n} A\rightarrow 0\] induces a map \begin{equation}\label{finitekernel}\op{Sel}(A[\varpi^n]/\Q_{\infty})\rightarrow \op{Sel}(A/\Q_{\infty})[\varpi^n]\end{equation}with finite kernel. The result follows.
\end{proof}
\begin{Remark}
It follows from standard arguments that the map $\eqref{finitekernel}$ is a pseudo-isomorphism. However, for our purposes it suffices to show that the kernel is finite, thereby shortening the argument.
\end{Remark}
For $n>0$, $A[\varpi^n]$ is identified with $T/\varpi^n$ as is $A^{\pm}[\varpi^n]$ with $T^{\pm}/\varpi^n$. Recall that since it is assumed that $\rho_{T/n}$ is aligned, $A^+[\varpi]$ is $\op{G}_{\Q}$-stable and there is an odd character $\varphi_{1,n}$ such that $A^+[\varpi^n]$ is isomorphic to $\On(\varphi_{1,n})$. In order to analyze the residual Selmer group associated to $A[\varpi^n]$, we consider a Selmer group associated to $A^+[\varpi^n]=\On(\varphi_{1,n})$ via local conditions at each prime $v\in S$. Let $L$ be a number field contained in $\Q_{\infty}$. At each prime $v\in S$, set
\[\mathcal{H}_v(A^+[\varpi^n]/L^{\op{cyc}}):=\begin{cases}  \bigoplus_{\eta|v} \op{im}\left\{H^1(L^{\op{cyc}}_{\eta}, A^+[\varpi^n])\rightarrow H^1(\op{I}_{\eta}, A^+[\varpi^n]) \right\}&\text{ if }v\neq p,\\
0 &\text{ if } v=p.\end{cases}\] When $v\neq p$, $\eta$ runs through the finite set of primes of $L^{\op{cyc}}$ that lie above $v$. For $v\in S$, denote by $\mathcal{H}_v(A^+[\varpi^n]/\Q_{\infty})$ the direct limit with respect to restriction maps
\[\mathcal{H}_v(A^+[\varpi^n]/\Q_{\infty}):=\varinjlim_L \mathcal{H}_v(A^+[\varpi^n]/L^{\op{cyc}}),\]where $L$ runs through all number fields contained in $\Q_{\infty}$. The Selmer group $\op{Sel}(A^+[\varpi^n]/\Q_{\infty})$ is the kernel of the localization map
\begin{equation}\label{lambdaS}\lambda_{S}:H^1\left(\Q_S/\Q_{\infty},A^+[\varpi^n]\right)\rightarrow \bigoplus_{v\in S} \mathcal{H}_v(A^+[\varpi^n]/\Q_{\infty}).\end{equation}
It is clear that the inclusion
\[A^+[\varpi^n]\hookrightarrow A[\varpi^n]\]induces a map of mod-$\varpi^n$ Selmer groups
\begin{equation}\label{modpin}\iota:\op{Sel}(A^+[\varpi^n]/\Q_{\infty})\rightarrow \op{Sel}(A[\varpi^n]/\Q_{\infty}).\end{equation}

\begin{Proposition}\label{prop33}
Suppose that $\Lambda_n(\op{G})$-corank of $\op{Sel}(A^+[\varpi^n]/\Q_{\infty})$ is positive. Then $\mu(A/\Q_{\infty})\geq n$ and $t(A/\Q_{\infty})\geq n$.
\end{Proposition}
\begin{proof}
The short exact sequence of $\op{G}_{\Q,S}$-modules
\[0\rightarrow A^+[\varpi^n]\rightarrow A[\varpi^n]\rightarrow A^-[\varpi^n]\rightarrow 0\] induces a long exact sequence in cohomology. Since $H^0(\Q_S/\Q_{\infty}, A^-[\varpi^n])$ is finite, the kernel of the map $\iota$ from $\eqref{modpin}$ is finite. The result follows from Proposition $\ref{prop32}$.
\end{proof}
The next lemma will be applied to analyzing the $\Lambda_n(\op{G})$-corank of the residual Selmer group $\op{Sel}(A^+[\varpi^n]/\Q_{\infty})$. Let $\op{G}_{\infty}=\op{Gal}(\mathbb{C}/\mathbb{R})$ denote the decomposition group at $\infty$.
\begin{Lemma}\label{globalECLemma}Let $M$ be a $\mathcal{O}[\op{G}_{\Q,S}]$-module which is a cofinitely generated and cofree $\mathcal{O}$-module. Then, we have that
\[\begin{split}&\operatorname{corank}_{\Lambda(\op{G})}H^1(\Q_S/\Q_{\infty}, M)\\=&\operatorname{corank}_{\Lambda(\op{G})}H^2(\Q_S/\Q_{\infty}, M)+\op{corank}_{\mathcal{O}} M- \op{corank}_{\mathcal{O}} H^0(\op{G}_{\infty}, M).\end{split}\]
\end{Lemma}
\begin{proof}
We have that
\[\begin{split}&\sum_{i\geq 0}(-1)^i \operatorname{corank}_{\Lambda(\op{G})}\left(H^{i}(\Q_S/\Q_{\infty},M)\right)\\=&\sum_{i,j\geq 0}(-1)^{i+j}\op{corank}_{\mathcal{O}}\left( H^j(\op{G}, H^{i}(\Q_S/\Q_{\infty},M))\right)\\=&\sum_{i\geq 0}(-1)^{i+1}\op{corank}_{\mathcal{O}}\left( H^{i}(\op{G}_{\Q,S},M)\right)
\\=&\op{corank}_{\mathcal{O}} M- \op{corank}_{\mathcal{O}} H^0(\op{G}_{\infty}, M).\end{split}\]
The first equality follows from \cite[Theorem 1.1]{HowsonEC}, the second from the inflation-restriction sequence (also known as the Hochschild-Serre spectral sequence) and the final equality follows from the global Euler-characteristic formula (see \cite[8.7.4]{NSW}) applied to each of the submodules $M[\varpi^m]$ as $m$ ranges over positive integers.
\end{proof}

\begin{Lemma}\label{globalECLemma2}Let $\kappa_n:\op{G}_{\Q,S}\rightarrow \op{GL}_1(\mathcal{O}_{/n})$ be a liftable odd continuous character and $n\geq 1$. Then, we have that:
\[\operatorname{corank}_{\Lambda_n(\op{G})}H^1(\Q_S/\Q_{\infty}, \On (\kappa_n))\geq 1.\]
\end{Lemma}
\begin{proof}
It follows from the assumption on $\kappa_n$ that it lifts to a character \[\kappa: \op{G}_{\Q,S}\rightarrow \op{GL}_1(\mathcal{O}),\] where $S$ is a finite set of prime numbers. There is a short exact sequence of $\op{G}_{\Q,S}$-modules 
\[0\rightarrow \On(\kappa_n)\rightarrow K/\mathcal{O} (\kappa)\xrightarrow{\varpi^n} K/\mathcal{O}(\kappa)\rightarrow 0 \] induces a surjective map 
\[H^1(\Q_S/\Q_{\infty}, \On (\kappa_n))\rightarrow H^1(\Q_S/\Q_{\infty}, K/\mathcal{O} (\kappa))[\varpi^n].\] The result follows from Lemma $\ref{globalECLemma}$, which asserts that:
\[\op{corank}_{\Lambda(\op{G})}H^1(\Q_S/\Q_{\infty}, K/\mathcal{O} (\kappa))\geq 1.\]
\end{proof}
\begin{Lemma}\label{lemma36}
Let $v\in S\backslash \{p\}$ and $M$ an $\F[\op{G}_v]$-module which is finite dimensional $\F$-vector space. Let $\eta|v$ be a prime of $\Q_{\infty}$ and denote by $\op{G}_{\eta}$ the absolute Galois group of $\Q_{\infty, \eta}$. Then, we have that 
\[\op{corank}_{\F[[\op{G}_{\eta}]]} H^1(\op{G}_{\eta}, M)=0.\]
\end{Lemma}
\begin{proof}
Note that since $\op{G}_{\eta}$ has $p$-cohomological dimension $\leq 1$, it follows that $H^i(\Q_{\infty,\eta}, M))=0$ for $i\geq 2$. The proof is similar to Lemma $\ref{globalECLemma}$ and the argument is a direct application of \cite[Proposition 1.6]{HowsonEC} and the local Euler characteristic formula.
\end{proof}

\begin{Lemma}\label{lemma37}
Assume that $\rho_{T/n}$ is aligned. For $v\in S$, the group $\mathcal{H}_v(A^+[\varpi^n]/\Q_{\infty})^{\vee}$ is a pseudonull $\Lambda(\op{G})$-module.
\end{Lemma}
\begin{proof}
Recall that $\mathcal{H}_p(A^+[\varpi^n]/\Q_{\infty})=0$. Therefore, assume without loss of generality that $v\neq p$. Note that $\mathcal{H}_v(A^+[\varpi^n]/\Q_{\infty})$ is a quotient of the product 
\[\widetilde{\mathcal{H}}_v(A^+[\varpi^n]/\Q_{\infty}):=\prod_{w|v} H^1(\Q_{\infty, w}, A^+[\varpi]),\] which is easier to analyze. Note that submodules of pseudonull modules are pseudonull. Therefore, it suffices to show that $\widetilde{\mathcal{H}}_v(A^+[\varpi^n]/\Q_{\infty})$ is a copseudonull $\Lambda(\op{G})$-module. The short exact sequence \[0\rightarrow A^+[\varpi^{n-1}]\rightarrow A^+[\varpi^n] \rightarrow A^+[\varpi]\rightarrow 0\]induces a three term exact sequence
\[ \widetilde{\mathcal{H}_v}(A^+[\varpi^{n-1}]/\Q_{\infty})\rightarrow \widetilde{\mathcal{H}_v}(A^+[\varpi^n]/\Q_{\infty})\rightarrow \widetilde{\mathcal{H}}_v(A^+[\varpi]/\Q_{\infty}).\]It suffices to show that the two flanking groups are copseudonull. Therefore, the d\'evissage argument shows that it suffices to prove the assertion for $n=1$. Recall that $\Omega(\op{G}):=\F[[\op{G}]]$. It suffices to show that $\widetilde{\mathcal{H}}_v(A^+[\varpi]/\Q_{\infty})$ is cotorsion as an $\Omega(\op{G})$-module. Pick a prime $\eta|v$ of $\Q_{\infty}$.   Note that $A^+[\varpi]\simeq \F(\varphi_1)$ and that there are isomorphisms
\[ \begin{split}
& \widetilde{\mathcal{H}}_v(A^+[\varpi]/\Q_{\infty})\\
\simeq &\op{Ind}_{\op{G}_{\eta}}^{\op{G}}\left( H^1(\Q_{\infty,\eta}, \F(\varphi_1))\right) \\ \simeq & \F[[\op{G}]]\otimes_{\F[[\op{G}_{\eta}]]} H^1(\Q_{\infty,\eta}, \F(\varphi_1)).\end{split}\] In the above formula, $\op{G}_{\eta}:=\op{Gal}(\Q_{\infty,\eta}/\Q_v)$, is identified with the decomposition group of $\eta$. The assertion follows from Lemma $\ref{lemma36}$, which states that $H^1(\Q_{\infty,\eta}, \F(\varphi_1))$ is a cotorsion $\F[[\op{G}_{\eta}]]$-module.
\end{proof}

\begin{Corollary}\label{corollory38}
Assume that $\rho_{T/n}$ is aligned and $\varphi_{1,n}$ is liftable, then \[\op{corank}_{\Lambda_n(\op{G})}\op{Sel}(A^+[\varpi^n]/\Q_{\infty})\geq 1.\]
\end{Corollary}
\begin{proof}
 Lemma $\ref{globalECLemma2}$ asserts that $H^1(\Q_S/\Q_{\infty}, A^+[\varpi^n])$ has $\Lambda_n(\op{G})$-corank $\geq 1$ and Lemma $\ref{lemma37}$ asserts that $\mathcal{H}_v(A^+[\varpi^n]/\Q_{\infty})^{\vee}$ is a pseudonull $\Lambda(\op{G})$-module, and hence is torsion as a $\Lambda_n(\op{G})$-module. Therefore, the Selmer group $\op{Sel}(A^+[\varpi^n]/\Q_{\infty})$ has $\Lambda_n(\op{G})$-corank $\geq 1$.
\end{proof}
\begin{proof}[Proof of Theorem $\ref{theorem1}$]
\par Since $\bar{\rho}$ is aligned, $\varphi_1$ is odd, and hence so is $\varphi_{1,n}$. Note that $A^+[\varpi^n]$ is identified with  $\mathcal{O}(\varphi_{1,n})$. Corollory $\ref{corollory38}$ asserts that the Selmer group $\op{Sel}(A^+[\varpi^n]/\Q_{\infty})$ has $\Lambda_n(\op{G})$-corank $\geq 1$. The result follows from Proposition $\ref{prop33}$.

\end{proof}

\section{On the {\large$\mu$}-multiplicity}\label{section 4}
In this section, Theorem $\ref{theorem2}$ is proved.
\begin{Lemma}\label{mumultbasiclemma}
Suppose that the $\Omega(\op{G})$-corank of $\op{Sel}(A[\varpi]/\Q_{\infty})$ is $r$. Then the $\mu$-multiplicity $\rm{r}(A/\Q_{\infty})$ is equal to $r$.
\end{Lemma}
\begin{proof}
It follows from the structure theory of $\Lambda(\op{G})$-modules that \[\rm{r}(A/\Q_{\infty})=\op{corank}_{\Omega(\op{G})}\op{Sel}(A/\Q_{\infty})[\varpi],\] see \eqref{strtheory}. There is a natural map 
\[\alpha:\op{Sel}(A[\varpi]/\Q_{\infty})\rightarrow \op{Sel}(A/\Q_{\infty})[\varpi]\]with finite kernel. It suffices to show that the cokernel is $\Omega(\op{G})$-cotorsion. Consider the diagram:
\[
\begin{tikzcd}[column sep = small, row sep = large]
0\arrow{r} & \op{Sel}(A[\varpi]/\Q_{\infty}) \arrow{r}  \arrow{d}{\alpha} & H^1(\Q_S/\Q_{\infty}, A[\varpi]) \arrow{r} \arrow{d}{\beta} & \bigoplus_{v\in S}\mathcal{H}_v(A[\varpi]/\Q_{\infty})  \arrow{d}{\gamma} \\
0\arrow{r} & \op{Sel}(A/\Q_{\infty})[\varpi] \arrow{r} & H^1(\Q_S/\Q_{\infty}, A[\varpi^{\infty}])[\varpi]  \arrow{r}  &\bigoplus_{v\in S}\mathcal{H}_v(A/\Q_{\infty})[\varpi],
\end{tikzcd}\]
We show that the kernel of $\gamma$ is $\Omega(\op{G})$-cotorsion. For $v\neq p$, Lemma \ref{lemma37} asserts that $\mathcal{H}_v(A[\varpi]/\Q_{\infty})$ is $\Omega(\op{G})$-torsion. The map 
\[\mathcal{H}_p(A[\varpi]/\Q_{\infty})\rightarrow \mathcal{H}_p(A/\Q_{\infty})[\varpi]\]is injective (cf. for instance, the proof of \cite[Proposition 2.8]{greenbergvatsal}). Therefore, the kernel of $\gamma$ is $\Omega(\op{G})$-torsion and this completes the proof.
\end{proof}
In order to analyze the group $\op{Sel}(A[\varpi]/\Q_{\infty})$, we analyze the image of the map \[\op{Sel}(A[\varpi]/\Q_{\infty})\longrightarrow H^1(\Q_S/\Q_{\infty}, \F(\varphi_2)).\] Let $S(\Q_{\infty})$ denote the set of primes of $\eta$ of $\Q_{\infty}$ that divide a prime $v\in S$. This image is contained in the Selmer group defined by:
\[\op{Sel}(\F(\varphi_2)/\Q_{\infty}):=\left\{f\in H^1(\Q_S/\Q_{\infty}, \F(\varphi_2))\mid f\text{ is unramified at all } \eta\in S(\Q_{\infty})\right\}.\]

\begin{Lemma}\label{psudonull lemma}
The group $\op{Sel}(\F(\varphi_2)/\Q_{\infty})^{\vee}$ is a pseudonull $\Lambda(\op{G})$-module.
\end{Lemma}
\begin{proof}
Since the order of the group $\Delta:=\op{Gal}(\Q_{\infty}(\varphi_2)/\Q_{\infty})$ is coprime to $p$, we have that $H^i(\Delta, \F(\varphi_2))=0$ for $i>0$. It follows from the inflation-restriction sequence that the restriction map induces an isomorphism
\[H^1(\Q_S/\Q_{\infty}, \F(\varphi_2))\xrightarrow{\sim} \op{Hom}\left(\op{Gal}(\Q_S/\Q_{\infty}(\varphi_2)),\F(\varphi_2)\right)^{\Delta}.\] Therefore, there is an isomorphism:
\begin{equation}\label{phi2}\op{Sel}(\F(\varphi_2)/\Q_{\infty})\xrightarrow{\sim} \op{Hom}\left(X_{\infty}/pX_{\infty},\F(\varphi_2)\right)^{\Delta}.\end{equation}
Note that $X_{\infty}/pX_{\infty}$ decomposes into eigenspaces for the action of $\Delta$
\[X_{\infty}/pX_{\infty}=\bigoplus_{\psi: \Delta\rightarrow \F_p^{\times}} \left(X_{\infty}/pX_{\infty}\right)^{\psi}.\]
Note that for $\psi\neq \varphi_2$, 
\[\op{Hom}\left(\left(X_{\infty}/pX_{\infty}\right)^{\psi},\F(\varphi_2)\right)^{\Delta}=0.\]
It follows from \eqref{phi2} that 
\[\begin{split}\op{Sel}(\F(\varphi_2)/\Q_{\infty})& \xrightarrow{\sim} \bigoplus_{\psi}\op{Hom}\left(\left(X_{\infty}/pX_{\infty}\right)^{\psi},\F(\varphi_2)\right)^{\Delta}\\
& \xrightarrow{\sim} \op{Hom}\left(\left(X_{\infty}/pX_{\infty}\right)^{\varphi_2},\F(\varphi_2)\right)^{\Delta}\\
&\xrightarrow{\sim} \op{Hom}\left(\left(X_{\infty}/pX_{\infty}\right)^{\varphi_2},\F\right).\end{split}\]
Since $X_{\infty}$ is a cotorsion $\Lambda$-module with $\mu(X_{\infty})=0$, it follows that $X_{\infty}/p$ is pseudonull as a $\Lambda(\op{G})$-module from which it follows that $\op{Sel}(\F(\varphi_2)/\Q_{\infty})^{\vee}$ is pseudonull as well.
\end{proof}
\begin{proof}[Proof of Theorem $\ref{theorem2}$]
We prove the result in the case when $\bar{\rho}_T$ is residually aligned. When $\bar{\rho}_T$ is skew, the argument is very similar. Since Theorem $\ref{theorem1}$ asserts that the $\mu$-invariant is positive, it remains to show that the $\mu$-multiplicity $\rm{r}(A/\Q_{\infty})\leq 1$. Consider the long exact sequence in cohomology associated to the short exact sequence of $\F[\op{G}_{\Q, S}]$-modules
\[0\rightarrow \F(\varphi_1)\rightarrow A[\varpi]\rightarrow \F(\varphi_2)\rightarrow 0.\]This induces an exact sequence
\[W_1\xrightarrow{\alpha_1} \op{Sel}(A[\varpi]/\Q_{\infty})\xrightarrow{\alpha_2} W_2\rightarrow 0.\]
Here, $W_1$ is the subgroup of $H^1(\Q_S/\Q_{\infty}, \F(\varphi_1))$ which is the preimage of the residual Selmer group $\op{Sel}(A[\varpi]/\Q_{\infty})$ w.r.t. the map
\[H^1(\Q_S/\Q_{\infty}, \F(\varphi_1))\rightarrow H^1(\Q_S/\Q_{\infty}, A[\varpi]). \]The group $W_2$ is the image of $\op{Sel}(A[\varpi]/\Q_{\infty})$ in $H^1(\Q_S/\Q_{\infty}, \F(\varphi_2))$. Lemma $\ref{globalECLemma2}$ asserts that $H^1(\Q_S/\Q_{\infty}, \F(\varphi_1))$ has $\Omega(\op{G})$-corank $1$. Therefore, the image of $\alpha_1$ has $\Omega(\op{G})$-corank $\leq 1$. We show that $W_2$ is $\Omega(\op{G})$-cotorsion. Recalling the Selmer conditions on $A[\varpi]$, it is easy to see that $W_2$ is contained in $\op{Sel}(\F(\varphi_2)/\Q_{\infty})$. By Lemma \ref{psudonull lemma}, it follows that $W_2$ is $\Omega(\op{G})$-cotorsion. Therefore, the $\Omega(\op{G})$-corank of $\op{Sel}(A[\varpi]/\Q_{\infty})$ is equal to $\leq 1$. The result follows from Lemma $\ref{mumultbasiclemma}$.

\par Similar arguments show that the $\mu$-multiplicity is $\leq 1$ in the residually skew case.
\end{proof}
\section{Lifting reducible Galois representations}\label{section 5}
\par Recall that $\F$ is a finite field of characteristic $p$ and that $\op{W}(\F)$ is the ring of Witt vectors with residue field $\F$. Let $\bar{\rho}:\op{G}_{\Q}\rightarrow \op{GL}_2(\F)$ be a reducible Galois representation and $\varphi$ be the characters such that $\bar{\rho}\simeq \mtx{\varphi}{\ast}{0}{1}$. Assume that $\bar{\rho}$ satisfies the conditions of Theorem $\ref{mainlifting}$, in particular, $\varphi$ is odd. Recall that $\chi$ is the $p$-adic cyclotomic character and $\bar{\chi}$ denote its mod-$p$ reduction. For any character $\alpha:\op{G}_{\Q}\rightarrow \op{GL}_1(\F)$, we denote by $\tilde{\alpha}:\op{G}_{\Q}\rightarrow \op{GL}_1(\op{W}(\F))$ the Teichm\"uller lift of $\alpha$. Express $\varphi$ as a product $\bar{\chi} \eta$. Fix a positive integer $k\equiv 2 \mod{p^N(p-1)}$ and denote by $\psi$ the character $\chi^{k-1} \tilde{\eta}$ lifting $\varphi$. Here we recall that $\tilde{\eta}$ is the Teichm\"uller lift of $\eta$.
\begin{Th}\label{th51}
Let $\bar{\rho}$ be as above and $N$ a positive integer. Let $k$ be a positive integer such that $k\equiv 2 \mod{p^N(p-1)}$ and $\psi$ be the character defined as above. There exists a Galois representation 
\[\rho: \op{G}_{\Q}\rightarrow \op{GL}_2(\op{W}(\F))\] lifting $\bar{\rho}$ such that the following assertions hold
\begin{enumerate}
    \item the mod-$p^N$ reduction is of the form $\rho_N=\mtx{\psi_N}{\ast}{0}{1}$.
    \item $\rho$ is unramified away from finitely many primes, 
    \item $\det \rho=\psi$,
    \item $\rho$ is irreducible,
    \item the local representation is of the form \[\rho_{\restriction \op{G}_p}=\mtx{\chi_{\restriction \op{G}_p}^{k-1}\delta_1}{\ast}{0}{\delta_2},\] where $\delta_1,\delta_2:\op{G}_p\rightarrow \op{GL}_1(\op{W}(\F))$ are unramified characters. In other words, $\rho$ is ordinary at $p$.
\end{enumerate} Furthermore, the representation $\rho$ arises from a Hecke eigencuspform $f$ of weight $k$.
\end{Th} 

Theorem $\ref{mainlifting}$ follows from Theorems $\ref{theorem1}$ and $\ref{th51}$ and is proved in section $\ref{section 6}$. The proof of Theorem $\ref{th51}$ is fairly involved and will be completed in the next section. In the remainder of this section, we prove some preliminary results towards this end.
\par Let $S$ be the set of primes containing $p$ and the primes at which $\bar{\rho}$ is ramified.
\begin{Definition} Let $\mathcal{C}_{\op{W}(\F)}$ be the category of coefficient rings over $\text{W}(\F)$ with residue field $\F$. The objects of this category consist of local $\text{W}(\F)$-algebras $(R,\mathfrak{m})$ for which
      \begin{itemize}
          \item $R$ is complete and noetherian,
          \item $R/\mathfrak{m}$ is isomorphic to $\F$ as a $\text{W}(\F)$-algebra.
      \end{itemize}A morphism in this category is a map of local rings which is also a $\text{W}(\F)$-algebra homorphism. \end{Definition}
      A coefficient ring $R$ is equipped with a residue map $R\rightarrow \F$, this is the composite of the reduction map $R\rightarrow R/\mathfrak{m}$ and the unique isomorphism $R/\mathfrak{m}\rightarrow \F$ of $\op{W}(\F)$-algebras. At each prime number $v$ choose an embedding $\iota_v:\bar{\Q}\hookrightarrow \bar{\Q}_v$ and note that $\iota_v$ fixes an inclusion $\op{G}_{v}\hookrightarrow \op{G}_{\Q}$. 
      
      \begin{Definition}
      Let $\Pi$ denote $\op{G}_{\Q}$ (resp. $\op{G}_{v}$), $\bar{r}:\Pi\rightarrow \op{GL}_2(\F)$ denote $\bar{\rho}$ (resp. $\bar{\rho}_{\restriction v}$) and $\nu$ denote $\psi$ (resp. $\psi_{\restriction v}$). Let $R$ be a coefficient ring with maximal ideal $\mathfrak{m}$. Note that $R$ is a $\op{W}(\F)$-algebra and we may view $\eta$ as a character
      \[\eta:\Pi\rightarrow \op{GL}_1(R).\]A \textit{lift} $r:\Pi\rightarrow \op{GL}_2(R)$ of $\bar{r}$ is a homomorphism which reduces to $\bar{r}$, as depicted
      \[ \begin{tikzpicture}[node distance = 2.5 cm, auto]
            \node at (0,0) (G) {$\Pi$};
             \node (A) at (3,0){$\op{GL}_2(\F)$.};
             \node (B) at (3,2){$\op{GL}_2(R)$};
      \draw[->] (G) to node [swap]{$\bar{r}$} (A);
       \draw[->] (B) to node{} (A);
      \draw[->] (G) to node {$r$} (B);
\end{tikzpicture}\]
      Furthermore, we stipulate that \[\op{det} r=\nu.\]
      \end{Definition}
      
      For $(R, \mathfrak{m})\in \mathcal{C}_{\op{W}(\F)}$, let $\widehat{\op{GL}_2}(R)\subset \op{GL}_2(R)$ be the subgroup of matrices which reduce to the identity modulo the maximal ideal.
      \begin{Definition} Let $\Pi$ denote $\op{G}_{\Q}$ (resp. $\op{G}_{v}$), two lifts $\rho,\rho':\Pi\rightarrow \op{GL}_2(R)$ of $\bar{\rho}$ (resp. $\bar{\rho}_{\restriction v}$) the residual representation are \textit{strictly-equivalent} if $\rho=A\rho' A^{-1}$ for $A\in \widehat{\op{GL}}_2(R)$. A \textit{deformation} is a strict equivalence class of lifts.
      \end{Definition}
     Note that if $\rho$ is a deformation of $\bar{\rho}$ then $\rho_{\restriction v}$ is a deformation of $\bar{\rho}_{\restriction v}$. Therefore, a global deformation gives rise to a local deformation at each prime $v$. At each prime $v$, set $\operatorname{Def}_v(R)$ to be the set of $R$-deformations of $\bar{\rho}_{\restriction v}$. The association $R\mapsto \operatorname{Def}_v(R)$ defines a functor $\operatorname{Def}_v:\mathcal{C}_{\op{W}(\F)}\rightarrow \operatorname{Sets}$. A deformation functor at $v$ is a subfunctor $\mathcal{C}_v$ of $\op{Def}_v$. For each coefficient ring $R$, the set $\mathcal{C}_v(R)$ is a subset of $\op{Def}_v(R)$. We say that a deformation $\varrho:\op{G}_v\rightarrow \op{GL}_2(R)$ of $\bar{\rho}_{\restriction v}$ \textit{satisfies} $\mathcal{C}_v$ if $\varrho\in \mathcal{C}_v(R)$. The lifting strategy is based on a local-global principle, where local deformation functors shall be fixed and the global deformations shall be required to satisfy these functors at the primes at which they are allowed to ramify.
     \begin{Definition}\label{adgaloisaction}Set $\op{Ad}\bar{\rho}$ to denote the Galois module whose underlying vector space consists of $2\times 2$ matrices with entries in $\F$. The Galois action is as follows: for $g\in \op{G}_{\Q}$ and $v\in \op{Ad}\bar{\rho}$, 
     set $g\cdot v:=\bar{\rho}(g) v \bar{\rho}(g)^{-1}$. Let $\g$ be the $\op{G}_{\Q}$-stable submodule of trace zero matrices.
     \end{Definition}
      Note that the ring of dual numbers $\F[\epsilon]/(\epsilon^2)$ is contained in $\mathcal{C}_{\op{W}(\F)}$. Let $v$ be a prime number, the functor of local deformations $\op{Def}_v$ is equipped with a tangent space. As a set, this is defined to be $\operatorname{Def}_v(\F[\epsilon]/(\epsilon^2))$. The tangent space has the structure of a vector space over $\F$ and is in bijection with $H^1(\op{G}_{v},\g)$. The bijection
      \[ H^1(\op{G}_{v},\g)\longleftrightarrow\operatorname{Def}_v(\F[\epsilon]/(\epsilon^2))\]identifies a cohomology class $f$ with the deformation \[(\operatorname{Id}+\epsilon f)\bar{\rho}_{\restriction v}: \operatorname{G}_{v}\rightarrow \op{GL}_2(\F[\epsilon]/(\epsilon^2)).\] Let $m\geq 1$, $\varrho_m\in \operatorname{Def}_v(\text{W}(\F)/p^{m})$ and $f\in H^1(\op{G}_{v}, \g)$. The \textit{twist} of $\varrho_m$ by $f$ is defined by the formula $(\op{Id}+p^{m-1}f)\varrho_m$. The following is a simple exercise left to the reader.
      \begin{Fact}\label{ptorsor}The set of deformations $\varrho_m$ of a fixed $\varrho_{m-1}\in \operatorname{Def}_v(\text{W}(\F)/p^{m-1})$ is either empty or in bijection with $H^1(\op{G}_{v},\g)$.
      \end{Fact} Let $R\in \mathcal{C}_{\op{W}(\F)}$ with maximal ideal $\mathfrak{m}_R$, and $I$ an ideal in $R$. The mod-$I$ reduction map $R\rightarrow R/I$ is said to be a \textit{small extension} if $I.\mathfrak{m}_R=0$.\begin{Definition}\label{localdefconditions}
A functor of deformations $\mathcal{C}_v:\mathcal{C}_{\op{W}(\F)}\rightarrow \op{Sets}$ of $\bar{\rho}_{\restriction v}$ is called a \textit{deformation condition} if the following conditions $\eqref{dc1}$ to $\eqref{dc3}$ stated below are satisfied.
      \begin{enumerate}
          \item\label{dc1} $\mathcal{C}_v(\F)=\{\bar{\rho}_{\restriction v}\}.$
          \item\label{dc2} For $i=1,2$, let $R_i\in \mathcal{C}_{\op{W}(\F)}$ and $\rho_i\in\mathcal{C}_v(R_i)$. Let $I_1$ be an ideal in $R_1$ and $I_2$ an ideal in $R_2$ such that there is an isomorphism $\alpha:R_1/I_1\xrightarrow{\sim} R_2/I_2$ satisfying \[\alpha(\rho_1 \;\text{mod}\;{I_1})=\rho_2 \;\text{mod}\;{I_2}.\] Let $R_3$ be the fibred product \[R_3=\lbrace(r_1,r_2)\mid \alpha(r_1\;\text{mod}\; I_1)=r_2\; \text{mod} \;I_2\rbrace\] and $\rho_1\times_{\alpha} \rho_2$ the induced $R_3$-representation, then $\rho_1\times_{\alpha} \rho_2\in \mathcal{C}_v(R_3)$.
          \item\label{dc3} Let $R\in \mathcal{C}_{\op{W}(\F)}$ with maximal ideal $\mathfrak{m}_R$. If $\rho\in \mathcal{C}_v(R)$ and $\rho\in \mathcal{C}_v(R/\mathfrak{m}_R^n)$ for all $n>0$ it follows that $\rho\in \mathcal{C}_v(R)$. In other words, the functor $\mathcal{C}_v$ is continuous.
      \end{enumerate}
      The functor of deformations $\mathcal{C}_v$ is said to be \textit{liftable} if for every small extension $R\rightarrow R/I$ the induced map $\mathcal{C}_v(R)\rightarrow \mathcal{C}_v(R/I)$ is surjective.
      \end{Definition}
   
Condition $\eqref{dc2}$ is referred to as the Mayer-Vietoris property. By a well-known result of Grothendieck \cite[section 18]{Mazurintro}, conditions $\eqref{dc1}$ to $\eqref{dc3}$ guarantee that $\mathcal{C}_v$ is pro-represented by a scheme $\op{Spec} R_v$, where $R_v\in  \mathcal{C}_{\op{W}(\F)}$. In other words, there is a deformation 
 \[\rho_v^{\op{univ}}:\op{G}_{v}\rightarrow \op{GL}_2(R_v)\]of $\bar{\rho}_{\restriction v}$ such that any deformation $\varrho\in \mathcal{C}_v(R)$ is induced by a unique map of coefficient rings $R_v\rightarrow R$. In addition, if it is liftable, then $\op{Spec} R_v$ is formally smooth.      
 \begin{Definition}The tangent space $\mathcal{N}_v$ of a deformation condition $\mathcal{C}_v$ consists of $f\in H^1(\op{G}_{v}, \g)$, such that $(\op{Id}+\epsilon f) \bar{\rho}_{\restriction v}\in \mathcal{C}_v(\F[\epsilon]/(\epsilon^2))$.
 \end{Definition}
 \begin{Fact} The action of $\mathcal{N}_v$ on $\operatorname{Def}_v(\text{W}(\F)/p^m)$ stabilizes $\mathcal{C}_v(\text{W}(\F)/p^m)$. In other words, if $\varrho_m\in \mathcal{C}_v(\text{W}(\F)/p^m)$ and $f\in \mathcal{N}_v$, then 
      \[(\op{Id}+p^{m-1}f)\varrho_{m}\in \mathcal{C}_v(\text{W}(\F)/p^m). \]
 \end{Fact}
     It follows from the work of Ramakrishna \cite{RaviFM} that at each prime $v\in S$, there is a suitable liftable deformation condition.
      \begin{Proposition}\label{existencecv}
      Let $v\in S$, there is a subfunctor $\mathcal{C}_v\subseteq \op{Def}_v$ such that
      \begin{enumerate}
          \item $\mathcal{C}_v$ is a liftable deformation condition,
          \item the dimension of the tangent space $\mathcal{N}_v$ is given by \[\dim  \mathcal{N}_v=\begin{cases}h^0(\op{G}_{v}, \g)& \text{ if }v\neq p,\\
          h^0(\op{G}_{v}, \g)+1&\text{ if }v=p.\end{cases}\]
      \end{enumerate} Moreover, $\mathcal{C}_p$ consists of ordinary deformations.
      \end{Proposition}
      \begin{proof}
      For the case $v\neq p$, the classification is due to Ramakrishna, see \cite[Proposition 1, Remark p.124]{RaviFM}. Since this will come up later in our construction, we recall the classification for $\mathcal{C}_v$ here. We have that $\bar{\rho}_{\restriction v}=\mtx{\varphi_v}{\ast}{0}{1}$, where $\varphi_v$ is the restriction of $\varphi$ to $\op{G}_v$. First, assume that $v\neq p$. Let $\psi_v$ denote the restriction of $\psi$ to $\op{G}_v$. As explained on \cite[p.120]{RaviFM} are two subcases to consider:
      \begin{enumerate}
          \item $\varphi_v$ is ramified at $v$, and in this case, the local residual representation $\bar{\rho}_{\restriction v}\simeq \mtx{\varphi_v}{0}{0}{1}$ decomposes into a direct sum of characters. Let $\mathcal{C}_v$ be the set of deformations of the form 
          \[\varrho=\mtx{\psi_v\gamma}{0}{0}{\gamma^{-1}}\]where $\gamma$ is an unramfied character lifting the trivial character.
          \item In the second case, $\varphi_v$ is unramified and the extension class $\ast \in H^1(\op{G}_v, \F(\varphi_v))$ is ramified. Note that 
          \[\dim H^1_{nr}(\op{G}_v, \F(\varphi_v))=\dim H^0(\op{G}_v, \F(\varphi_v)).\]On the other hand, it follows from the local Euler characteristic formula, and local duality that
          \[\begin{split}&\dim H^1(\op{G}_v, \F(\varphi_v))-\dim H^0(\op{G}_v, \F(\varphi_v))\\=&\dim H^2(\op{G}_v, \F(\varphi_v))\\
          =& \dim H^0(\op{G}_v, \F(\bar{\chi}\varphi_v^{-1})).\end{split}\]Hence, there will only exist a ramified cohomology class $\ast\in H^1(\op{G}_v, \g)$ if and only if 
          $\varphi_v=\bar{\chi}_{\restriction v}$. Note that $\bar{\rho}_{\restriction v}$ cuts out a tamely ramified extension of $\Q_v$ and factors through the maximal tamely ramified extension of $\Q_v$. This quotient is generated by Frobenius $\sigma_v$ and tame inertia generator $\tau_v$, subject to the relation $\sigma_v\tau_v \sigma_v^{-1}=\tau_v^v$. Set $\mathcal{C}_v$ to consist of deformations $\varrho$ with 
          \[\varrho(\sigma_v)=(\psi(\sigma_v)v^{-1})^{\frac{1}{2}}\mtx{v}{x}{0}{1}\text{ and }\varrho(\tau_v)=\mtx{1}{y}{0}{1},\] for suitable $x,y\in \op{W}(\F)$.
      \end{enumerate}
      The deformation condition $\mathcal{C}_p$ consists of $\varrho$ such that there is an unramified character $\gamma$ such that
      \[\varrho=\mtx{\psi_v\gamma}{\ast}{0}{\gamma^{-1}}.\]These are referred to as ordinary deformations, moreover, the tangent space $\mathcal{N}_p$ has dimension $\dim  \mathcal{N}_p=h^0(\op{G}_{p}, \g)+1$. The reader may refer to \cite[Corollory 4.5]{PatEx} for more properties of ordinary deformations.
      \end{proof} 
     \par In order to lift $\bar{\rho}$ to a characteristic zero representation, we allow the deformations to be ramified not only at the set of primes $S$, but also at a finite auxiliary set of primes $X$ disjoint from $S$. At each prime $v\in X$, Hamblen and Ramakrishna define a liftable deformation functor $\mathcal{C}_v$. However, this deformation functor $\mathcal{C}_v$ is not prorepresentable. As a result, there is no natural notion of a tangent space at $v$. However, following \cite{hamblenramakrishna}, we may specify a subspace $\mathcal{N}_v\subseteq H^1(\op{G}_v, \g)$. Before we make these definitions we introduce the notion of a \textit{highly versal pair}:
     \begin{Definition}
     Let $v$ be a prime not contained in $S$. Let $\mathcal{C}_v$ be a deformation functor and $\mathcal{N}_v$ a subspace of $H^1(\op{G}_{v}, \g)$. For an integer $m>1$, the pair $(\mathcal{C}_v, \mathcal{N}_v)$ is a highly versal pair of degree $m$ if:
     \begin{enumerate}
         \item for all integers $k\geq m$, $f\in \mathcal{N}_v$ and $\varrho\in \mathcal{C}_v(\op{W}(\F)/p^k)$, the twist $(\op{Id}+p^{k-1}f) \varrho$ is contained in $\mathcal{C}_v(\op{W}(\F)/p^k)$. In other words, the action of $\mathcal{N}_v$ on $\op{Def}_v(\op{W}(\F)/p^k)$ stabilizes $\mathcal{C}_v(\op{W}(\F)/p^k)$.
         \item The integer $m$ is the smallest for which the above property is satisfied. In other words, the action of $\mathcal{N}_v$ on $\op{Def}_v(\op{W}(\F)/p^{m-1})$ does not stabilize $\mathcal{C}_v(\op{W}(\F)/p^{m-1})$.
     \end{enumerate}
     \end{Definition}The name "highly versal" is due to Ramakrishna, as the first named author recalls from past conversations. The auxiliary primes $v$ we work with were introduced in \cite[section 4]{hamblenramakrishna} and are referred to as \textit{trivial primes}. We recall the definition here.
     \begin{Definition}\label{trivialprimes}
     A prime number $v\notin S$ is a trivial prime if:
     \begin{enumerate}
         \item $\bar{\rho}_{\restriction v}$ is the trivial representation (i.e., $\bar{\rho}_{\restriction v}(g)=\op{Id}$ for all $g\in \op{G}_v$),
         \item $v\equiv 1\mod{p}$ and $v\not\equiv 1\mod{p^2}$.
     \end{enumerate}
     \end{Definition}
     
    At a trivial prime $v$, we introduce four choices of highly versal pairs $(\mathcal{C}_v, \mathcal{N}_v)$ of degree $3$. Let $v$ be a trivial prime. All deformations of $\bar{\rho}_{\restriction v}$ factor through the Galois group of the maximal pro-$p$ extension of $\Q_v$ over $\Q_v$. This Galois group is generated by a Frobenius element $\sigma_v$ and a noncanonical generator $\tau_v$ of pro-$p$ (tame) inertia. These elements satisfy a single relation $\sigma_v \tau_v \sigma_v^{-1}=\tau_v^v$. Specifying a deformation $\varrho$ of $\bar{\rho}_{\restriction v}$ amounts to specifying elements $\varrho(\sigma_v)$ and $\varrho(\tau_v)$ in $\widehat{\op{GL}_2}$ such that $\varrho(\sigma_v) \varrho(\tau_v) \varrho(\sigma_v)^{-1}=\varrho(\tau_v)^v$. In what follows, the square root of $v$ is the square root which is $\equiv 1\mod{p}$. For $R\in \mathcal{C}_{\op{W}(\F)}$, let $\mathcal{D}_v(R)\subset \op{Def}_v(R)$ consist of deformations with a representative $\varrho:\op{G}_{v}\rightarrow \op{GL}_2(R)$ satisfying:
   \begin{equation}
       \varrho(\sigma_v)=(\psi(\sigma_v)v^{-1})^{\frac{1}{2}}\mtx{v}{x}{0}{1}\text{ and } \varrho(\tau_v)=\mtx{1}{y}{0}{1},
   \end{equation}
 where $x$ and $y$ are in the maximal ideal $\mathfrak{m}_R$ of $R$. The factor $(\psi(\sigma_v)v^{-1})^{\frac{1}{2}}$ ensures that the determinant of $\varrho(\sigma_v)$ is equal to $\psi(\sigma_v)$. Note that since $v\equiv 1\mod{p}$, $\varrho$ is a lift of $\bar{\rho}_{\restriction v}$.
 \par Note that since $v\not\equiv 1\mod{p^2}$, it follows that $(v-1)\neq 0$ when viewed as an element of $\F$. Following \cite{hamblenramakrishna}, define elements $f_1, f_2, g^{\op{nr}},g^{\op{ram}}\in H^1(\op{G}_{v}, \g)$ as follows:
\[
\begin{split}
& f_1(\sigma_v)=\mtx{0}{1}{0}{0} \text{, }f_1(\tau_v)=\mtx{0}{0}{0}{0},\\ &
f_2(\sigma_v)=\mtx{0}{0}{0}{0} \text{, }f_2(\tau_v)=\mtx{0}{1}{0}{0},\\ & g^{\operatorname{nr}}(\sigma_v)=\mtx{0}{0}{1}{0} \text{, }g^{\operatorname{nr}}(\tau_v)=\mtx{0}{0}{0}{0},\\ & g^{\operatorname{ram}}(\sigma_v)=\mtx{0}{0}{1}{0} \text{, }g^{\operatorname{ram}}(\tau_v)=\mtx{-\frac{y}{v-1}}{0}{0}{\frac{y}{v-1}}.
\end{split}\]
\par The space $\mathcal{Q}_v$ denotes the subspace of $H^1(\op{G}_{v}, \g)$ spanned by $\{f_1,f_2\}$, let $\mathcal{P}_v^{\operatorname{nr}}$ the subspace spanned by $\{f_1,f_2,g^{\operatorname{nr}}\}$ and $\mathcal{P}_v^{\operatorname{ram}}$ the subspace spanned by $\{f_1,f_2,g^{\operatorname{ram}}\}$. The functor $\mathcal{D}_v^{\op{nr}}$ (resp. $\mathcal{D}_v^{\op{ram}}$) was introduced in \cite[Proposition 24]{hamblenramakrishna} (resp. \cite[Proposition 28]{hamblenramakrishna}). The deformation functor $\mathcal{D}_v^{\op{nr}}$ consists of deformations satisfying $\mathcal{D}_v$ for which $p^2|x,y$. On the other hand, $\mathcal{D}_v^{\op{ram}}$ consists of deformations satisfying $\mathcal{D}_v$ for which $p^2|x$ and $p||y$. Suppose that $A\in \op{GL}_2(\op{W}(\F))$ is an invertible matrix, and $\varrho$ is a deformation of $\bar{\rho}_{\restriction v}$, then so is the conjugate $A \varrho A^{-1}$. This is due to $\bar{\rho}_{\restriction v}$ being the trivial representation. We define four choices of deformation functors at a trivial prime $v$.
\begin{Definition}\label{cvdeftrivial}
Let $v$ be a trivial prime and $\mathcal{D}_v^{\op{nr}}$, $\mathcal{D}_v^{\op{ram}}$, $\mathcal{P}_v^{\op{nr}}$, $\mathcal{P}_v^{\op{ram}}$ be defined as in the previous paragraph. Then consider the four pairs $(\mathcal{C}_v, \mathcal{N}_v)$:
\begin{enumerate}
    \item type 1: following \cite[Definition 26]{hamblenramakrishna}, let $\mathcal{C}_v$ (resp. $\mathcal{N}_v$) consist of conjugates by $\mtx{1}{0}{1}{1}$ of elements of $\mathcal{D}_v^{\op{nr}}$ (resp. $\mathcal{P}_v^{\op{nr}}$).
    \item type 2: following \cite[Definition 30]{hamblenramakrishna}, let $\mathcal{C}_v$ (resp. $\mathcal{N}_v$) consist of conjugates by $\mtx{0}{1}{1}{0}$ of elements of $\mathcal{D}_v^{\op{ram}}$ (resp. $\mathcal{P}_v^{\op{ram}}$).
    \item type 3: let $\mathcal{C}_v$ (resp. $\mathcal{N}_v$) denote $\mathcal{D}_v^{\op{nr}}$ (resp. $\mathcal{P}_v^{\op{nr}}$).
    \item type 4: let $\mathcal{C}_v$ (resp. $\mathcal{N}_v$) denote $\mathcal{D}_v^{\op{ram}}$ (resp. $\mathcal{P}_v^{\op{ram}}$).
\end{enumerate}
We shall impose the conditions $(\mathcal{C}_v, \mathcal{N}_v)$ for which the type may range from $1$ to $4$ depending on the specific application.

\begin{Lemma}\label{highlyversallemma}
Let $v$ be a trivial prime and $(\mathcal{C}_v, \mathcal{N}_v)$ of one of the types $(1)-(4)$. Then, it is a highly versal pair of degree $3$.

\end{Lemma}
\begin{proof}
The result is a consequence of \cite[Corollary 25,29]{hamblenramakrishna}.
\end{proof}
\begin{Lemma}
Let $v$ be a trivial prime and $(\mathcal{C}_v, \mathcal{N}_v)$ of one of the types $(1)-(4)$. Then, $\mathcal{C}_v$ is a liftable deformation functor and $\op{dim}\mathcal{N}_v=h^0(\op{G}_v, \g)$.
\end{Lemma}
\begin{proof}
Since $\bar{\rho}_{\restriction v}$ is the trivial representation, $h^0(\op{G}_v, \g)=3$, on the other hand, $\dim \mathcal{N}_v=3$ as well. It is easy to show that $\mathcal{D}_v^{\op{nr}}$ and $\mathcal{D}_v^{\op{ram}}$ are both liftable deformation functors, we leave this to the reader. It follows from this that $\mathcal{C}_v$ is liftable as well.
\end{proof}
\end{Definition}

     Let $X$ be a finite set of trivial primes disjoint from $S$ and assume that at each prime $v\in X$ we have chosen one of the four choices for $(\mathcal{C}_v, \mathcal{N}_v)$. For $v\in S\cup X$, set $\mathcal{N}_v^{\perp}\subseteq H^1(\op{G}_v,\g^*)$ to be the orthogonal complement of $\mathcal{N}_v$ with respect to the non-degenerate local Tate pairing 
\[H^1(\op{G}_v, \g)\times H^1(\op{G}_v, \g^*)\rightarrow H^2(\op{G}_v, \F(\bar{\chi}))\xrightarrow{\sim}\F.\]
\begin{Definition}The \textit{Selmer condition} $\mathcal{N}$ is the tuple $\{\mathcal{N}_v\}_{v\in S\cup X}$ and the \textit{dual Selmer condition} $\mathcal{N}^{\perp}$ is $\{\mathcal{N}_v^{\perp}\}_{v\in S\cup X}$. Attached to $\mathcal{N}$ and $\mathcal{N}^{\perp}$ are the \textit{Selmer and dual-Selmer groups} defined as follows:
\begin{equation}\begin{split}
    & H^1_{\mathcal{N}}(\op{G}_{\Q,S\cup X}, \g)\\
    :=& \text{ker}\left\{ H^1(\operatorname{G}_{\Q, S\cup X}, \g)\xrightarrow{\operatorname{res}_{S\cup X}} \bigoplus_{v\in S\cup X} \frac{H^1(\op{G}_v, \g)}{\mathcal{N}_v}\right\}\\
    & H^1_{\mathcal{N}^{\perp}}(\op{G}_{\Q,S\cup X}, \g^*)\\
    :=& \text{ker}\left\{ H^1(\operatorname{G}_{\Q, S\cup X}, \g^{*})\xrightarrow{\op{res}_{S\cup X}'} \bigoplus_{v\in S\cup X} \frac{H^1(\op{G}_v, \g^*)}{\mathcal{N}_v^{\perp}}\right\}
\end{split}\end{equation}
      respectively. 
      \end{Definition}
      The Selmer and dual Selmer groups fit a Poitou-Tate sequence. We only point out that the cokernel of $\op{res}_{S\cup X}$ injects into $H^1_{\mathcal{N}^{\perp}}(\op{G}_{\Q,S\cup X}, \g^*)^{\vee}$. In particular, if the dual Selmer group is zero, the restriction map $\operatorname{res}_{S\cup X}$ is surjective. Since the spaces $\mathcal{N}_v$ at a trivial prime $v$ stabilize lifts only past mod $p^3$, it becomes necessary to produce a mod $p^3$ lift $\rho_3$ of $\bar{\rho}$ before applying the general lifting-strategy.
      
      \section{Modular deformations with large $\mu$-invariant}\label{section 6}
      \par In this section, Theorem $\ref{th51}$ is proved. Let $\bar{\rho}$ be a reducible Galois representation satisfying the conditions of $\ref{th51}$. As mentioned in the previous section, we may assume without loss of generality that $\bar{\rho}=\mtx{\varphi}{\ast}{0}{1}$. Let $\mathfrak{n}$ and $\mathfrak{b}$ denote the unipotent matrices $\mtx{0}{\ast}{0}{0}$ and upper triangular matrices $\mtx{\ast}{\ast}{0}{\ast}$ in $\g$ respectively. Note that since the image of $\bar{\rho}$ consists of upper triangular matrices, both $\mathfrak{n}$ and $\mathfrak{b}$ are $\op{G}_{\Q}$-stable submodules. For a character $\kappa:\op{G}_{\Q}\rightarrow \op{GL}_1(\op{W}(\F))$, let $\kappa_n$ denote the mod-$p^n$ reduction of $\kappa$. Note that $\kappa_1$ coincides with the mod-$p$ reduction $\bar{\kappa}$. 
      \par Set $\mathcal{M}$ to denote the modules $\{\mathfrak{n}, \mathfrak{b}, \g, \mathfrak{n}^*, \mathfrak{b}^*, \g^*\}$. For $M\in \mathcal{M}$, and a set of primes $Z$ containing $S$, denote by 
      \[\Sha_Z^i(M):=\op{ker}\left\{ H^i(\op{G}_{\Q, Z}, M)\rightarrow \bigoplus_{v\in Z} H^i(\op{G}_v, M)\right\}.\]Let $X$ be a finite set of primes disjoint from $S$, define a \textit{relative $\Sha$-group} as follows
      \[\Sha^i_{X,S\cup X}( M):=\op{ker}\left\{ H^i(\op{G}_{\Q, S\cup X}, M)\rightarrow \bigoplus_{v\in X} H^i(\op{G}_v, M)\right\}.\] Note that we have an exact sequence
      \[0\rightarrow \Sha^i_{S\cup X} (M)\rightarrow \Sha^i_{X,S\cup X} (M)\rightarrow \bigoplus_{v\in S} H^i(\op{G}_v, M).\]
      We shall first add a finite number of trivial primes to the level, so that certain $\Sha$-groups vanish. It will then be shown that certain globally defined cohomology classes vanish because they lie in these vanishing $\Sha$-groups. Recall that a prime $v\notin S$ is a trivial prime if it satisfies the conditions of Definition $\ref{trivialprimes}$.
      \begin{Proposition}\label{prop61} There is a finite set of trivial primes $X_0$ (disjoint from $S$) such that $\Sha^1_{X_0, S\cup X_0}(M)=0$ for all $M\in \mathcal{M}$. 
      \end{Proposition}
      \begin{proof}
      The result \cite[Proposition 13]{hamblenramakrishna} asserts that $\Sha^1_{S\cup X_0}(M)=0$ for all $M\in \mathcal{M}$. It follows verbatim from the argument that the set of primes $X_0$ can be chosen so that $\Sha^1_{X_0,S\cup X_0}(M)=0$ for all $M\in \mathcal{M}$. 
      \end{proof}
      \begin{Corollary}\label{cor62}
      Let $X_0$ be the set of primes as in Proposition $\ref{prop61}$, and $M\in \{\mathfrak{n}, \mathfrak{b}, \g\}$. Then, the localization map 
          \begin{equation}H^1(\op{G}_{\Q,S\cup X_0} M)\rightarrow \bigoplus_{v\in S_0} H^1(\op{G}_v, M)\end{equation} is surjective and $\Sha^2_{S\cup X_0}(M)=0$.
      \end{Corollary}
      \begin{proof}
      We define Selmer data $\mathcal{L}$ on the set of primes $S\cup X_0$ as follows 
      \[\mathcal{L}_v=\begin{cases}H^1(\op{G}_v, M) &\text{ if }v\in X_0,\\
      0 &\text{ if }v\in S_0.\\
      \end{cases}\]
      It is easy to see that the dual Selmer data is defined by 
      \[\mathcal{L}_v^{\perp}=\begin{cases}0 &\text{ if }v\in X_0,\\
      H^1(\op{G}_v, M^*) &\text{ if }v\in S_0.\\
      \end{cases}\]
      Consider the restriction map 
      \[\op{res}: H^1(\op{G}_{S\cup X_0}, M)\rightarrow \bigoplus_{v\in S\cup X_0} H^1(\op{G}_v, M)/\mathcal{L}_v.\] The Selmer and dual Selmer groups fit into a five-term Poitou-Tate sequence (see the proof of \cite[Lemma 1.1]{taylor}). We need only point out that the cokernel of $\op{res}$ injects into $H^1_{\mathcal{L}^{\perp}}(\op{G}_{\Q,S\cup X_0}, M^*)^{\vee}$.
      Note that 
      \[\bigoplus_{v\in S_0\cup X_0} H^1(\op{G}_v, M)/\mathcal{L}_v=\bigoplus_{v\in S_0} H^1(\op{G}_v, M)\] and that $H^1_{\mathcal{L}^{\perp}}(\op{G}_{\Q, S\cup X}, M^*)$ coincides with $\Sha^1_{X_0, S\cup X_0}(M)$. Note that the set of primes $X_0$ is chosen so that this group is equal to zero. Therefore, the restriction map
      \[\op{res}:H^1(\op{G}_{\Q, S\cup X_0}, M)\rightarrow \bigoplus_{v\in S_0} H^1(\op{G}_v, M)\] is surjective. On the other hand, by the well-known global duality theorem for $\Sha$-groups, $\Sha^2_{S\cup X_0}(M)$ is dual to $\Sha^1_{S\cup X_0}(M^*)$. The set of primes $X_0$ is chosen so that $\Sha^1_{X_0, S\cup X_0}(M^*)=0$. Since $\Sha^1_{X_0, S\cup X_0}(M^*)=0$, so is $\Sha^1_{S\cup X_0}(M^*)$.
      
      \end{proof}
      Write $\varphi$ as a product $\bar{\chi}\eta$, where we recall that $\bar{\chi}$ is the mod-$p$ cyclotomic character. Recall that $\psi$ is the chosen lift of $\varphi$ given by $\psi=\chi^{k-1} \tilde{\eta}$, where $\tilde{\eta}$ is the Teichm\"uller lift of $\eta$. For $n\geq 1$, let $\psi_n$ denote the reduction of $\psi$ modulo $p^n$. Recall that $N$ is the integer specified in Theorem $\ref{mainlifting}$. At each prime $v\in S$, the deformation condition $\mathcal{C}_v$ is as specified in Proposition $\ref{existencecv}$. On the other hand, there are four possible choices of deformation functors $\mathcal{C}_v$ for $v\in X_0$, see Definition $\ref{cvdeftrivial}$. In the result below, the deformation condition will be taken to be of type (3) or (4) depending on whether the lifts of $\bar{\rho}_{\restriction v}$ are unramified or ramified modulo $p^2$ respectively. For a coefficient ring let $\op{U}(R)$ denote the subgroup of $\op{GL}_2(R)$ consisting of matrices of the form $\mtx{1}{\ast}{0}{1}$. Recall that $\mathfrak{n}$ consists of the strictly upper triangular matrices in $\g$. Denote by $\iota$ and $\iota_v$ the natural maps
      \[\iota: H^1(\op{G}_{\Q}, \mathfrak{n})\rightarrow H^1(\op{G}_{\Q}, \g)\]
      \[\iota_v: H^1(\op{G}_{v}, \mathfrak{n})\rightarrow H^1(\op{G}_{v}, \g)\]induced by the inclusion $\mathfrak{n}\hookrightarrow \g$.
     
 \begin{Proposition}\label{prop63}
 Let $X_0$ be the set of primes chosen in Proposition $\ref{prop61}$. Then, the residual representation $\bar{\rho}$ lifts to a mod-$p^N$ Galois representation \[\rho_N:\op{G}_{\Q, S\cup X_0}\rightarrow \op{GL}_2(\op{W}(\F)/p^N)\] satisfying the following conditions:
          \begin{itemize}
              \item $\rho_N=\mtx{\psi_N}{\ast}{0}{1}$
              \item $\rho_N$ satisfies the condition $\mathcal{C}_v$ at each prime $v\in S\cup X_0$.
              \end{itemize}
             Set $\rho_2$ to be $\rho_N\mod{p^2}$. Then for $v\in X_0$, the pair $(\mathcal{C}_v, \mathcal{N}_v)$ is type (3) if $\rho_2$ is unramified at $v$ and type (4) if $\rho_2$ is ramified at $v$.
 \end{Proposition}
 \begin{proof}
The proof is by induction. Let $n<N$ and assume that $\bar{\rho}$ lifts to
\[\rho_n:\op{G}_{\Q,S\cup X_0}\rightarrow \op{GL}_2(\op{W}(\F)/p^n)\]$\rho_n=\mtx{\psi_n}{\ast}{0}{1}$ satisfying the conditions $\mathcal{C}_v$ at the primes $v\in S\cup X$. We show that $\rho_n$ may be lifted one more step to \[\rho_{n+1}:\op{G}_{\Q,S\cup X_0}\rightarrow \op{GL}_2(\op{W}(\F)/p^{n+1})\] such that $\rho_{n+1}=\mtx{\psi_{n+1}}{\ast}{0}{1}$, satisfying the same conditions.
\par We proceed in two steps.
\begin{enumerate}
    \item First, we show that $\rho_n$ lifts to a Galois representation $r_{n+1}$ of the form $r_{n+1}=\mtx{\psi_{n+1}}{\ast}{0}{1}$.
    \item Next, we replace $r_{n+1}$ by
    \[\rho_{n+1}:=(\op{Id}+p^n \iota(f)) r_{n+1}\]for a suitable cohomology class $f\in H^1(\op{G}_{\Q, S\cup X_0}, \mathfrak{n})$ such that it satisfies the constraints $\mathcal{C}_v$ at all primes $v\in S\cup X_0$. 
\end{enumerate} It is clear that there is a set theoretic lift \[\tau_{n+1}:\op{G}_{S\cup X_0}\rightarrow \op{GL}_2(\op{W}(\F)/p^{n+1})\] of the form $\tau_{n+1}=\mtx{\psi_{n+1}}{\ast}{0}{1}$. In other words, this lift is not necessarily a homomorphism but however, the diagonal entries are the specified characters. Note that for $g, h\in \op{G}_{\Q, S\cup X_0}$, \[\tau_{n+1}(gh) \tau_{n+1}(h)^{-1}\tau_{n+1}(g)^{-1}\] is a unipotent matrix which reduced to the identity modulo $p^n$. This function on $\op{G}_{\Q, S\cup X_0}\times\op{G}_{\Q, S\cup X_0}$ is trivial precisely when $\tau_{n+1}$ is a homomorphism. Identify $\mathfrak{n}$ with the kernel of the reduction map 
\[\mathfrak{n}=\op{ker}\left\{ \op{U}(\op{W}(\F)/p^{n+1})\rightarrow \op{U}(\op{W}(\F)/p^{n}) \right\}.\] Associate a cohomological obstruction to lifting as follows, set 
\[\mathcal{O}(\rho_n)(g,h):=\tau_{n+1}(gh)\tau_{n+1}(h)^{-1} \tau_{n+1}(g)^{-1}.\] This cocycle gives a well defined cohomology class $\mathcal{O}(\rho_n)\in H^2(\op{G}_{\Q, S\cup X_0}, \mathfrak{n})$. It is an easy exercise to show that the obstruction class $\mathcal{O}(\rho_n)$ is trivial if and only if there is a lift $r_{n+1}:\op{G}_{\Q, S\cup X_0}\rightarrow \op{GL}_2(\op{W}(\F)/p^{n+1})$ of $\rho_n$ such that 
\begin{enumerate}
    \item $r_{n+1}$ is not just a set theoretic lift, but also a homomorphism, 
    \item $r_{n+1}$ is of the form $\mtx{\psi_{n+1}}{\ast}{0}{1}$.
\end{enumerate} Note that each local representation $\rho_{n|v}$ satisfies the liftable deformation condition $\mathcal{C}_v$. As a result, the obstruction class $\mathcal{O}(\rho_n)$ becomes trivial when localized at a prime $v\in S\cup X_0$. As a result, $\mathcal{O}(\rho_n)$ lies in $\Sha^2_{S\cup X_0}(\g)$. On the other hand, $X_0$ is chosen so that $\Sha^2_{S\cup X_0}(\g)=0$ (according to Corollary $\ref{cor62}$). Therefore, the obstruction class is also trivial and $\rho_{n}$ lifts to $r_{n+1}$ satisfying the above conditions.
\par We now show that for a suitable cohomology class $f\in H^1(\op{G}_{S\cup X_0}, \mathfrak{n})$, the twist $\rho_{n+1}:=(\op{Id}+p^n \iota(f))r_{n+1}$ satisfies the local conditions $\mathcal{C}_v$ at all primes $v\in S\cup X_0$. First we show that for any cohomology class $f\in H^1(\op{G}_{S\cup X_0}, \mathfrak{n})$, the deformation $(\op{Id}+p^n \iota(f))r_{n+1}$ satisfies the condition $\mathcal{C}_v$ at each prime $v\in X_0$. Note that $r_{n+1}=\mtx{\psi_{n+1}}{\ast}{0}{1}$and that $\varphi(\sigma_v)=1$ since $\bar{\rho}(\sigma_v)=\op{Id}$ ($v$ is a trivial prime). Since $k\equiv 2\mod{p^N(p-1)}$, and $\psi=\chi^{k-1}\tilde{\eta}$, it follows that $\psi_{n+1}(\sigma_v)=v$. Therefore, \[r_{n+1}(\sigma_v)=\mtx{v}{x}{0}{1}\text{ and }r_{n+1}
(\sigma_v)=\mtx{1}{y}{0}{1}.\] Moreover, for any cohomology class $f\in H^1(\op{G}_{\Q, S\cup X_0}, \mathfrak{n})$, the twist $(\op{Id}+p^n \iota(f))r_{n+1}$ satisfies the same relations at a prime $v\in X_0$ (with $x$ and $y$ possibly replaced by some other elements $x'$ and $y'$). As a result, the twist $(\op{Id}+p^n f)r_{n+1}$ automatically satisfies $\mathcal{C}_v$. Here $\mathcal{C}_v$ is of type (3) or (4) depending on whether $\rho_{2}$ is unramified at $v$ or not.
\par On the other hand, at each prime $v\in S$, there is a cohomology class $f_v\in H^1(\op{G}_{v}, \mathfrak{n})$ such that the twist $(\op{Id}+\iota_v(f_v)) r_{n+1|v}$ satisfies $\mathcal{C}_v$. This follows from inspecting the two cases outlined in the proof of Proposition $\ref{existencecv}$. According to Corollary $\ref{cor62}$, the map
\[H^1(\op{G}_{\Q,S\cup X_0}, \mathfrak{n})\rightarrow \bigoplus_{v\in S} H^1(\op{G}_v, \mathfrak{n})\] is surjective, there is a global cohomology class $f$ such that $f_{\restriction v}=f_v$ for all $v\in S$. The twist $\rho_{n+1}:=(\op{Id}+p^n \iota(f))r_{n+1}$ satisfies all conditions $\mathcal{C}_v$. Furthermore, $\rho_{n+1}$ is also of the form $\mtx{\psi_{n+1}}{\ast}{0}{1}$ since $f$ is a cohomology class with values in $\mathfrak{n}$. This completes the induction step.
 \end{proof}
 
Before proving Theorem $\ref{th51}$, we briefly outline the strategy of the proof. The residual representation $\bar{\rho}$ is lifted to a characteristic zero representation
\[\rho:\op{G}_{\Q, S\cup X}\rightarrow \op{GL}_2(\op{W}(\F))\] satisfying the required conditions. This is done by showing that there is a compatible system of lifts 
\[\rho_n:\op{G}_{\Q, S\cup X}\rightarrow \op{GL}_2(\op{W}(\F)/p^n)\] for every integer $n>0$. Proposition $\ref{prop63}$ asserts that there is a lift $\rho_N$ of the form 
\[\rho_N=\mtx{\psi_N}{\ast}{0}{1}.\]
Next, $\rho_N$ is lifted to a representation $\rho_{N+1}$ whose image is no longer upper triangular. It is shown that 
$\rho_{N+1}$ can be chosen so that its image contains
\[\op{SL}_2^{N}:=\op{ker}\left\{\op{SL}_2(\op{W}(\F)/p^{N+1})\rightarrow \op{SL}_2(\op{W}(\F)/p^{N})\right\}.\] As a consequence, any characteristic zero lift of $\rho_{N+1}$ will be irreducible. At this stage, it is shown that $\rho_{N+1}$ may be lifted to $\rho$ one step at a time as depicted in the following diagram

\[\begin{tikzpicture}[node distance = 1.5 cm, auto]
      \node (GSX) at (0,0){$\op{G}_{\Q}$};
      \node (GL2) at (5,0){$\op{GL}_{2}(\op{W}(\F)/p^N).$};
      \node (GL2Wn) at (3,2)[above of= GL2]{$\op{GL}_{2}(\op{W}(\F)/p^{m})$};
      \node (GL2Wnplus1) at (5,4){$\op{GL}_{2}(\op{W}(\F)/p^{m+1})$};
      \draw[->] (GSX) to node [swap]{$\rho_{N+1}$} (GL2);
      \draw[->] (GL2Wn) to node {} (GL2);
      \draw[->] (GSX) to node [swap]{$\rho_m$} (GL2Wn);
      \draw[->] (GL2Wnplus1) to node {} (GL2Wn);
      \draw[dashed,->] (GSX) to node {$\rho_{m+1}$} (GL2Wnplus1);
      \end{tikzpicture}\] 
\begin{proof}[Proof of Theorem $\ref{th51}$]
The result is completely analogous to that of \cite{hamblenramakrishna}. The only difference lies in the fact that in our construction, $\bar{\rho}$ is first lifted to a representation of the form $\rho_N=\mtx{\psi_N}{\ast}{0}{1}$, before it is lifted to characteristic zero.
\par Assume without loss of generality that $N>1$. Proposition $\ref{prop63}$ asserts that $\bar{\rho}$ lifts to $\rho_N=\mtx{\psi_N}{\ast}{0}{1}$ satisfying the the local conditions $\mathcal{C}_v$ at the primes $v\in S\cup X_0$. We first show that $\rho_N$ lifts to a representation 
\begin{equation}\label{rhoprime}\rho_{N+1}':\op{G}_{\Q, S\cup X_0}\rightarrow \op{GL}_2(\op{W}(\F)/p^{N+1}).\end{equation} Let $\tau_{N+1}:\op{G}_{\Q, S\cup X_0}\rightarrow \op{GL}_2(\op{W}(\F)/p^{N+1})$ be a set theoretic lift of $\rho_N$ with determinant equal to $\psi_{N+1}$. In other words, $\tau_{N+1}$ is a lift which is not necessarily a homomorphism, however, the $\det \tau_{N+1}$ is equal to $\psi_{N+1}$. Such a lift $\tau_{N+1}$ clearly does exist. Identify $\g$ with the kernel of the mod-$p^N$ reduction map
\[\op{ker}\left\{\op{SL}_2(\op{W}(\F)/p^{N+1})\rightarrow \op{SL}_2(\op{W}(\F)/p^{N}) \right\}.\]
Here, $X\in \g$ is identified with $\op{Id}+p^N X$.
For $g, h\in \op{G}_{\Q, S\cup X_0}$, set
\[\mathcal{O}(\rho_N)(g,h):=\tau_{n+1}(gh) \tau_{n+1}(h)^{-1}\tau_{n+1}(g)^{-1}\in \g.\] This defines a cohomology class in $H^2(\op{G}_{\Q, S\cup X_0}, \g)$ which vanishes precisely when there is a representation $\rho_{N+1}'$ lifting $\rho_N$ as in $\eqref{rhoprime}$. Since $\rho_N$ satisfies a liftable deformation condition at each prime $v\in S\cup X_0$, it follows that at each prime $v\in S\cup X_0$, the restriction of $\mathcal{O}(\rho_N)$ to $\op{G}_v$ is trivial. As a result, this obstruction class lies in $\Sha^2_{S\cup X_0}(\g)$. The set $X_0$ is chosen so that $\Sha^2_{S\cup X_0}(\g)=0$. Hence, there is a representation
\[\rho_{N+1}':\op{G}_{\Q, S\cup X_0}\rightarrow \op{GL}_2(\op{W}(\F)/p^{N+1})\] lifting $\rho_N$.
\par Let $v\in S\cup X_0$, since $\mathcal{C}_v$ is a liftable deformation condition and $\rho_{N\restriction v}$ satisfies $\mathcal{C}_v$, there is a lift 
\[\varrho_v:\op{G}_v\rightarrow \op{GL}_2(\op{W}(\F)/p^{N+1})\]satisfying $\mathcal{C}_v$. Therefore, by Fact $\ref{ptorsor}$, there is a cohomology class $f_v\in H^1(\op{G}_v, \g)$ such that the twist $(\op{Id}+p^N f_v) \rho_{N+1|v}'$ is equal to $\varrho_v$ and hence, satisfies $\mathcal{C}_v$. At this stage, it follows from \cite[Theorem 41]{hamblenramakrishna} that there are two trivial primes $\{v_1, v_2\}$ disjoint from $S\cup X_0$ and a global cohomology class \[f\in H^1(\op{G}_{\Q, S\cup X_0\cup \{v_1, v_2\}}, \g)\] such that the representation 
\[\rho_{N+1}:=(\op{Id}+p^{N} f) \rho_{N+1}'\]
satisfies the local conditions $\mathcal{C}_v$ at the primes $v\in S\cup X_0\cup \{v_1, v_2\}$. At the primes $v_1$ and $v_2$, the deformation condition $\mathcal{C}_v$ is of type (2). It follows from the argument in the proof of \cite[Proposition 42]{hamblenramakrishna} that the image of $\rho_{N+1}$ contains 
\[\op{SL}_2^{N}:=\op{ker}\left\{\op{SL}_2(\op{W}(\F)/p^{N+1})\rightarrow \op{SL}_2(\op{W}(\F)/p^{N})\right\}.\] As a result, any characteristic zero lift 
\[\rho:\op{G}_{\Q}\rightarrow \op{GL}_2(\op{W}(\F))\] of $\rho_{N+1}$ must be irreducible.
\par Next, we show via an inductive argument that $\rho_{N+1}$ can be lifted to characteristic zero one step at a time. This argument essentially follows along the lines of the proof of \cite[Theorem 2]{hamblenramakrishna}. We provide a sketch here. It follows from \cite[Proposition 44]{hamblenramakrishna} that there is a finite set of trivial primes $X$ containing $X_0\cup \{v_1, v_2\}$ such that the Selmer and dual Selmer groups 
$H^1_{\mathcal{N}}(\op{G}_{\Q, S\cup X}, \g)$ and $H^1_{\mathcal{N}^{\perp}}(\op{G}_{\Q, S\cup X}, \g^*)$ are both zero. At a prime $v\in X\backslash (\{v_1, v_2\}\cup X_0)$, let $\mathcal{C}_v$ be the deformation functor of type (1). Since the Selmer and dual Selmer groups are zero, it follows by Poitou-Tate that the map
\begin{equation}\label{selmeriso}H^1(\op{G}_{\Q, S\cup X}, \g)\rightarrow \bigoplus_{v\in S\cup X} H^1(\op{G}_v, \g)/\mathcal{N}_v\end{equation} is an isomorphism. In fact, we shall only use surjectivity in what follows. Assume that for some integer $m\geq N+1$, there exists a lift 
\[\rho_m:\op{G}_{\Q, S\cup X}\rightarrow\op{GL}_2(\op{W}(\F)/p^m)\]satisfying the conditions $\mathcal{C}_v$ at all primes $v\in S\cup X$. Recall the argument from earlier in the proof to show that $\rho_{N}$ lifted to $\rho_{N+1}'$. This crucially required that $\rho_N$ satisfy liftable deformation conditions at the primes at which it ramifies. The same argument shows that $\rho_m$ lifts to a Galois representation
\[\rho_{m+1}':\op{G}_{\Q, S\cup X}\rightarrow \op{GL}_2(\op{W}(\F)/p^{m+1}).\] In order to complete the inductive step, we show that after twisting $\rho_{m+1}'$ by a suitable cohomology class, it can be arranged to satisfy the deformation conditions $\mathcal{C}_v$ at all primes $v\in S\cup X$. At each prime $v\in S\cup X$, there exists $z_v \in H^1(\op{G}_v, \g)$ such that the twist $(\op{Id}+p^m z_v)\rho_{m+1|v}'$ satisfies $\mathcal{C}_v$.
\par For $v\in S$, the space $\mathcal{N}_v$ is the tangent space of $\mathcal{C}_v$ and hence stabilizes $\mathcal{C}_v$ via the twisting action. On the other hand, for $v\in X$, the pair $(\mathcal{C}_v, \mathcal{N}_v)$ is a highly versal pair of degree $3$ (see Lemma $\ref{highlyversallemma}$). Since $N>1$, it follows that $m+1\geq 3$. Let $v\in S\cup X$, it follows that if $g_v\in \mathcal{N}_v$, then, the twist 
\[\left(\op{Id}+p^m(z_v+g_v)\right)\rho_{m+1|v}'\]also satisfies $\mathcal{C}_v$. Since the map $\eqref{selmeriso}$ is surjective, there is a cohomology class $z\in H^1(\op{G}_{\Q, S\cup X}, \g)$ such that \[z_{\restriction v}=z_v\mod{\mathcal{N}_v}\] for all $v\in S\cup X$. Then setting
\[\rho_{m+1}:=(\op{Id}+p^m z)\rho_{m+1}',\]find that $\rho_{m+1}$ satisfies the conditions $\mathcal{C}_v$ for $v\in S\cup X$. It follows from the inductive argument that there is a characteristic zero lift 
\[\rho:\op{G}_{\Q, S\cup X}\rightarrow \op{GL}_2(\op{W}(\F))\]lifting $\rho_{N+1}$. Recall that all lifts are stipulated to have determinant $\psi$. Moreover, as mentioned earlier in the proof, $\rho$ is necessarily irreducible. Furthermore, the characteristic zero representation satisfies the local conditions $\mathcal{C}_v$ at each prime $v\in S\cup X$. In particular, $\rho_{\restriction p}$ satisfies the ordinary deformation condition. It follows from the main result of Skinner and Wiles \cite{skinnerwiles} that $\rho$ arises from a Hecke eigencuspform.
\end{proof}
Next, we give the proof of Theorem $\ref{mainlifting}$.
\begin{proof}[Proof of Theorem $\ref{mainlifting}$] Let $V$ be the underlying space for the Galois representation associated to $f$. The integral representation $\rho$ is via the Galois action on a certain choice of Galois stable lattice $T$ in $V$. Here, the representation $\rho_N$ coincides with $\rho_{T/N}$ and since $\rho_N$ is of the form $\mtx{\psi_N}{\ast}{0}{1}$, it follows that $\rho_{T/N}$ is aligned. It follows from Theorem $\ref{theorem1}$ that the $p$-primary Selmer group has $\mu$-invariant $\geq N$.

\end{proof}

\section{Examples}\label{section 7}

\begin{Example}
This example clarifies the difference between the residually aligned and residually skew cases, and illustrates Theorem $\ref{theorem1}$. All computations are aided by Sage and consultation of data from LMFDB. Set $p=5$, the Selmer groups in this example will be considered over $\Q^{\op{cyc}}$, the cyclotomic $\Z_5$-extension of $\Q$. Consider the three elliptic curves of Cremona label 11a, namely, $E_1=11a1$, $E_2=11a2$ and $E_3=11a3$. There is an isogeny each pair of elliptic curves of degree dividing $25$. Let 
\[\rho_{E_i}: \op{G}_{\Q}\rightarrow \op{GL}_2(\Z_5)\] be the Galois representation on the $5$-adic Tate module of $E_i$. In all three cases, the residual representation is reducible and the semisimplification of $\bar{\rho}_{E_i}$ equals \[\bar{\rho}_{E_i}^{\op{ss}}\simeq \mtx{\bar{\chi}}{0}{0}{1}.\] However the representations $\bar{\rho}_{E_i}$ are all distinct. Let $\mu_{E_i}$ denote the $\mu$-invariant of the $5$-primary Selmer group associated to $E_i$. The representation $\bar{\rho}_{E_1}$ is decomposes into a direct sum of characters
\[\bar{\rho}_{E_1}\simeq \mtx{\bar{\chi}}{0}{0}{1}.\] Since the mod-$p$ cyclotomic character $\bar{\chi}$ is odd, it is easy to see that $\bar{\rho}_{E_1}$ is \textit{aligned}. Direct computation shows that the $\mu$-invariant $\mu_{E_1}=1$. On the other hand, Theorem $\ref{theorem1}$ asserts that $\mu_{E_1}\geq 1$.
\par The representation $\bar{\rho}_{E_2}$ is indecomposable of the form 
\[\bar{\rho}_{E_2}\simeq \mtx{\bar{\chi}}{\ast}{0}{1},\]
hence, it is aligned. Theorem $\ref{theorem1}$ asserts that $\mu_{E_2}\geq 1$. Direct computation shows that $\mu_{E_2}=2$.
\par The representation $\bar{\rho}_{E_2}$ is indecomposable of the form 
\[\bar{\rho}_{E_3}\simeq \mtx{1}{\ast}{0}{\bar{\chi}}.\] It is easy to see that the residual representation is skew. Direct computation shows that $\mu_{E_3}=0$.
The authors have inspected similar examples which show that when the residual representation is skew, the $\mu_{E_3}=0$. Since any representation which is aligned is isogenous to one which is skew, this provides further evidence for Greenberg's conjecture.
\end{Example}
\begin{Example}
We show that there are many examples of Galois representations $\bar{\rho}$ satisfying the conditions of Theorem $\ref{mainlifting}$. Such examples are constructed in \cite[Lemma 3.4]{ray1}. Let $\op{Cl}(\Q(\mu_p))$ denote the class group of $\Q(\mu_p)$ and set $\mathcal{C}:=\op{Cl}(\Q(\mu_p))\otimes \F_p$. As a Galois module, $\mathcal{C}$ decomposes into eigenspaces
\[\mathcal{C}=\bigoplus_{i=0}^{p-2} \mathcal{C}(\bar{\chi}^i).\] Assume that $p\geq 5$ and there is an odd integer $2\leq i\leq p-3$ such that $\mathcal{C}(\bar{\chi}^i)\neq 0$. Then according to \textit{loc. cit.}, associated to any $\F_p$-quotient of $\mathcal{C}(\bar{\chi}^i)$, there is a residual representation $\bar{\rho}=\mtx{\bar{\chi}^i}{\ast}{0}{1}$ satisfying the conditions of \cite[Theorem 2]{hamblenramakrishna}. These are precisely, the conditions we impose in Theorem $\ref{mainlifting}$. The reader is also referred to \cite[section 7]{hamblenramakrishna} where more examples are computed.
\end{Example}


\begin{thebibliography}{1}
\bibitem{bellaichepollack}J. Bella\"iche, R. Pollack. \textit{Congruences with Eisenstein series and mu-invariants.} arXiv preprint arXiv:1806.04240 (2018).

\bibitem{coateshowson}J.H. Coates, S. Howson. \textit{Euler characteristics and elliptic curves II.} Journal of the Mathematical Society of Japan 53.1 (2001): 175-235.


\bibitem{drinen}M.J. Drinen, \textit{Finite submodules and Iwasawa $\mu$-invariants.} Journal of Number Theory 93.1 (2002): 1-22.

\bibitem{ferrerowashington}B. Ferrero, L.C. Washington. \textit{The Iwasawa invariant $\mu_p$ vanishes for abelian number fields.} Annals of Mathematics (1979): 377-395.


\bibitem{greenbergIT} R. Greenberg. \textit{Iwasawa theory for p-adic representations.} Algebraic Number Theory—in Honor of K. Iwasawa. Mathematical Society of Japan, 1989.

\bibitem{greenbergIWEC}R. Greenberg. \textit{Iwasawa theory for elliptic curves.} Arithmetic theory of elliptic curves. Springer, Berlin, Heidelberg, 1999. 51-144.

\bibitem{greenbergvatsal}R. Greenberg and V. Vatsal, \textit{On the Iwasawa invariants of elliptic curves}, Invent. math. 142.1 (2000): 17-63.


\bibitem{hamblenramakrishna} S. Hamblen, R. Ramakrishna. \textit{Deformations of certain reducible {G}alois representations {II}}. In:Amer. J. Math.130.4(2008),pp.913-944.

\bibitem{HowsonEC}S. Howson, \textit{Euler characteristics as invariants of Iwasawa modules.} Proceedings of the London Mathematical Society 85.3 (2002): 634-658.

\bibitem{Howsoncentraltorsion}S. Howson, \textit{Structure of central torsion Iwasawa modules.} Bulletin de la Société Mathématique de France 130.4 (2002): 507-535.

\bibitem{mflim} M.F.Lim \textit{Comparing the $\pi$-primary submodules of the dual Selmer groups.} Asian J. Math. 21 (2017), no. 6, 1153–1181.

\bibitem{mazur1}B. Mazur, \textit{Rational points of abelian varieties with values in towers of number fields.} Inventiones mathematicae 18.3-4 (1972): 183-266.

\bibitem{Mazurintro}
B. Mazur, \textit{An introduction to the deformation theory of Galois representations.} Modular forms and Fermat’s last theorem. Springer, New York, NY, 1997. 243-311.

\bibitem{NSW}J. Neukirch, A. Schmidt, K. Wingberg, \textit{Cohomology of number fields.} Vol. 323. Springer Science and Business Media, 2013.

\bibitem{perbet}G.Perbet, \textit{On Iwasawa invariants in p-adic Lie extensions.} Algebra \& Number Theory 5.6 (2012): 819-848.

\bibitem{PatEx}S.Patrikis, \textit{Deformations of Galois representations and exceptional monodromy.} Inventiones mathematicae 205.2 (2016): 269-336.

\bibitem{Ram2} R. Ramakrishna, \textit{Lifting Galois representations.} Inventiones mathematicae 138.3 (1999): 537-562.

\bibitem{RaviFM}R. Ramakrishna, \textit{Deforming Galois representations and the conjectures of Serre and Fontaine-Mazur.} Annals of mathematics 156.1 (2002): 115-154.

\bibitem{ray1}A. Ray, \textit{Constructing certain special analytic Galois extensions.} Journal of Number Theory 212 (2020): 105-112.

\bibitem{raysujatha}A. Ray, R. Sujatha. \textit{Euler Characteristics and their Congruences in the Positive Rank Setting.} Canadian Mathematical Bulletin (2019): 1-19.

\bibitem{skinnerwiles}C.M. Skinner, A.J. Wiles. \textit{Residually reductible representations and modular forms.} Publications Mathématiques de l'IH\'ES 89 (1999): 5-126.

\bibitem{taylor}R. Taylor, \textit{On icosahedral Artin representations, II.} American journal of mathematics 125.3 (2003): 549-566.

\bibitem{venjakob}O. Venjakob, \textit{On the structure theory of the Iwasawa algebra of a p-adic Lie group.} Journal of the European Mathematical Society 4.3 (2002): 271-311.

\bibitem{venjakobmu}O. Venjakob, \textit{On the Iwasawa theory of p-adic Lie extensions.} Compositio Mathematica 138.1 (2003): 1-54.

\end{thebibliography}
\end{document}